\documentclass[a4paper,11pt]{article}

\usepackage{latexsym}
\usepackage{amssymb}
\usepackage{amsmath}
\usepackage{color}
\usepackage{graphicx}
\usepackage{amsfonts}
\usepackage{amscd}
\usepackage{vmargin}
\usepackage[english]{babel}
\usepackage[latin1]{inputenc}

\newcommand{\p}{\mathbb{P}}

\newcommand{\R}{\mathbb{R}}

\newcommand{\N}{\mathbb{N}}
\newcommand{\E}{\mathbb{E}}
\newcommand{\x}{\mathbb{X}}
\newcommand{\M}{\mathcal{M}}

\newcommand{\pen}{\textrm{pen}}

\newcommand{\Var}{\textrm{Var}}

\newtheorem{theo}{Theorem}[section]
\newtheorem{coro}[theo]{Corollary}
\newtheorem{prop}[theo]{Proposition}
\newtheorem{lemma}[theo]{Lemma}

\newenvironment{proof}[2][Proof]{ 
\begin{trivlist} 
\item[\hskip \labelsep {\bfseries \itshape \textit{#1  #2}}]~\\ \indent}
{\end{trivlist}}

\begin{document}
\title{Optimal model selection in density estimation}
\date{}
\author{Matthieu Lerasle\footnote{IME-USP, granted by FAPESP Processo 2009/09494-0}}
\maketitle
\hspace{1cm}\begin{minipage}{12cm}
\begin{center}
{\bf Abstract}
\end{center}
{\small We build penalized least-squares estimators using the slope heuristic and resampling penalties. We prove oracle inequalities for the selected estimator with leading constant asymptotically equal to $1$. We compare the practical performances of these methods in a short simulation study.}
\end{minipage}
\vspace{0.5cm}

\noindent {\bf Key words:} Density estimation, optimal model selection, resampling methods, slope heuristic.

\noindent {\bf 2000 Mathematics Subject Classification:} 62G07, 62G09.

\section{Introduction}
The aim of model selection is to construct data-driven criteria to select a model among a given list. The history of statistical model selection goes back at least to Akaike \cite{Ak70}, \cite{Ak73} and Mallows \cite{Ma73}. They proposed to select among a collection of parametric models the one which minimizes an empirical loss plus some penalty term proportional to the dimension of the model. Birg\'e $\&$ Massart \cite{BM97} and Barron, Birg\'e $\&$ Massart \cite{BBM99} generalized this approach, making in particular the link between model selection and adaptive estimation. They proved that previous methods, in particular cross-validation (see Rudemo \cite{Ru82}) and hard thresholding (see Donoho {\it et.al.} \cite{DJKP}) can be viewed as penalization methods. More recently, Birg\'e $\&$ Massart \cite{BM07}, Arlot $\&$ Massart \cite{AM08} and Arlot \cite{Ar08}, (see also \cite{Ar07})  arised the problem of optimal efficient model selection. Basically, the aim is to select an estimator satisfying an oracle inequality with leading constant asymptotically equal to $1$. They obtained such procedures thanks to a sharp estimator of the ideal penalty $\pen_{id}$. We will be interested in two natural ideas, that are used in practice to evaluate $\pen_{id}$ and proved to be efficient in other frameworks. The first one is the slope heuristic. It was introduced in Birg\'e $\&$ Massart \cite{BM07} in Gaussian regression and developed in Arlot $\&$ Massart \cite{AM08} in a $M$-estimation framework. It allows to optimize the choice of a leading constant in the penalty term, provided that we know the shape of $\pen_{id}$. The other one is Efron's resampling heuristic. The basic idea comes from Efron \cite{Ef83} and was used by Fromont \cite{Fr04} in the classification framework. Then, Arlot \cite{Ar08} made the link with ideal penalties and developed the general procedure. Up to our knowledge, these methods have only been theoretically validated in regression frameworks. We propose here to prove their efficiency in density estimation. Let us now explain more precisely our context.

\subsection{Least-squares estimators}
In this paper, we define and study efficient penalized least-squares estimators in the density estimation framework when the error is measured with the $L^2$-loss. We observe $n$ i.i.d random variables $X_1,...,X_n$, defined on a probability space $(\Omega,\mathcal{A},\p)$, valued in a measurable space $(\x,\mathcal{X})$, with common law $P$. We assume that a measure $\mu$ on $(\x,\mathcal{X})$ is given and we denote by $L^2(\mu)$ the Hilbert space of square integrable real valued functions defined on $\x$. $L^2(\mu)$ is endowed with its classical scalar product, defined for all $t,t'$ in $L^2(\mu)$ by
$$<t,t'>=\int_{\x}t(x)t'(x)d\mu(x)$$
and the associated $L^2$-norm $\|.\|$, defined for all $t$ in $L^2(\mu)$ by $\|t\|=\sqrt{<t,t>}$. The parameter of interest is the density $s$ of $P$ with respect to $\mu$, we assume that it belongs to $L^2(\mu)$. The risk of an estimator $\hat{s}$ of $s$ is measured with the $L^2$-loss, that is $\|s-\hat{s}\|^2$, which is random when $\hat{s}$ is.\\
$s$ minimizes the integrated quadratic contrast $PQ(t)$, where $Q:L^2(\mu)\rightarrow L^1(P)$ is defined for all $t$ in $L^2(\mu)$ by $Q(t)=\|t\|^2-2t$. Hence, density estimation is a problem of $M$-estimation. These problems are classically solved in two steps. First, we choose a "model" $S_m$ that should be close to the parameter $s$, which means that $\inf_{t\in S_m}\|s-t\|^2$ is "small". Then, we minimize over $S_m$ the empirical version of the integrated contrast, that is, we choose 
\begin{equation}\label{LSE}
\hat{s}_m\in\arg\min_{t\in S_m} P_nQ(t). 
\end{equation}
This last minimization can be computationaly untractable for general sets $S_m$, leading to untractable procedures in practice. However, it can be easily solved when $S_m$ is a linear subspace of $L^2(\mu)$ since, for all orthonormal basis $(\psi_{\lambda})_{\lambda\in m}$, 
\begin{equation}\label{hsm}
\hat{s}_m=\sum_{\lambda\in m}(P_n\psi_{\lambda})\psi_{\lambda}.
\end{equation}
Thus, we will always assume that a model is a linear subspace in $L^2(\mu)$. The risk of the least-squares estimator $\hat{s}_m$ defined in (\ref{LSE}) is then decomposed in two terms, called bias and variance, thanks to Pythagoras relation. Let $s_m$ be the orthogonal projection of $s$ onto $S_m$,
$$\|s-\hat{s}_m\|^2=\|s-s_m\|^2+\|s_m-\hat{s}_m\|^2.$$ 
The statistician should choose a space $S_m$ realizing a trade-off between those terms. $S_m$ must be sufficiently ``large'' to ensure a small bias $\|s-s_m\|^2$, but not too much, for the variance $\|s_m-\hat{s}_m\|^2$ not to explose. The best trade-off depends on unknown properties of $s$, since the bias is unknown, and on the behavior of the empirical minimizer $\hat{s}_m$ in the space $S_m$. Classically, $S_m$ is a parametric space and the dimension $d_m$ of $S_m$ as a linear space is used to give upper bounds on $D_m=n\E\left(\|s_m-\hat{s}_m\|^2\right)$. This approach is validated in regular models under the assumption that the support of $s$ is a known compact, as mentioned in section \ref{Ch2S3}. However, this definition can fail dramatically because there exist simple models (histograms with a small dimension $d_m$) where $D_m$ is very large, and infinite dimensional models where $D_m$ is easily upper bounded. This issue is extensively discussed in Birg\'e \cite{Bi08}. Birg\'e chooses to keep the dimension $d_m$ of $S_m$ as a complexity measure and build new estimators that achieve better risk bounds than the empirical minimizer. His procedures are unfortunatly untractable for the practical user because he can only prove the existence of his estimators. Even his bounds on the risk are only interesting theoretically because they involve constants which are not optimal. We will not take this point of view here and our estimator will always be the empirical minimizer, mainly because it can easily be computed, see (\ref{hsm}). We will focus on the quantity $D_m/n$ and introduce a general Assumption (namely Assumption {\bf[V]}) that allows to work indifferently with $D_m/n$ or with the actual risk $\|s_m-\hat{s}_m\|^2$. We will also provide and study an estimator of $D_m/n$ based on the resampling heuristic.\\
We insist here on the fact that, unlike classical methods, we will not use in this paper strong extra assumptions on $s$, like $\left\|s\right\|_{\infty}<\infty$ or assume that $s$ is compactly supported.

\subsection{Model selection}
Recall that the choice of an optimal model $S_m$ is impossible without strong assumptions on $s$, for example a precise information on its regularity. However, under less restrictive hypotheses, we can build a countable collection of models $(S_m)_{m\in\M_n}$, growing with the number of observations, such that the best estimator in the associated collection $(\hat{s}_m)_{m\in\M_n}$ realizes an optimal trade-off, see for example Birg\'e $\&$ Massart \cite{BM97} and Barron, Birg\'e $\&$ Massart \cite{BBM99}. The aim is then to build an estimator $\hat{m}$ such that our final estimator, $\tilde{s}=\hat{s}_{\hat{m}}$ behaves almost as well as any model $m_o$ in the set of oracles
$$\M_n^*=\{m_o\in\M_n,\;\|\hat{s}_{m_o}-s\|^2=\inf_{m\in\M_n}\|\hat{s}_m-s\|^2\}.$$
This is the problem of model selection. More precisely, we want that $\tilde{s}$ satisfies an oracle inequality defined in general as follows.
\vspace{0.1cm}\\
{\bf Definition:} {\it (Trajectorial oracle inequality) Let $(p_n)_{n\in \N}$ be a summable sequence and let $(C_n)_{n\in \N}$ and $(R_{m,n})_{n\in \N}$ be sequences of positive real numbers. The estimator $\tilde{s}=\hat{s}_{\hat{m}}$ satisfies a trajectorial oracle inequality $TO(C_n,(R_{m,n})_{m\in\M_n},p_n)$ if
\begin{equation}\label{OI}
\forall n\in \N^*,\;\p\left(\|\tilde{s}-s\|^2> C_n\inf_{m\in\M_n}\left\{\|s-\hat{s}_{m}\|^2+R_{m,n}\right\}\right)\leq p_n.
\end{equation}
When $\tilde{s}$ satisfies an oracle inequality, $C_n$ is called the leading constant.}
\vspace{0.1cm}\\
In this paper, we are interested in the problem of optimal model selection defined as follows.
\vspace{0.1cm}\\
{\bf Definition:} {\it (Optimal model selection) We say that $\tilde{s}$ is optimal or that the procedure of selection $(X_1,...,X_n)\mapsto \hat{m}$ is optimal when $\tilde{s}$ satisfies a trajectorial oracle inequality $TO(1+r_n,(R_{m,n})_{m\in\M_n},p_n)$ with $r_n\rightarrow 0$ and for all $n$ in $\N^*$ and $m$ in $\M_n$ $R_{m,n}=0$. In order to simplify the notations, when $\tilde{s}$ is optimal we will say that $\tilde{s}$ satisfies an optimal oracle inequality $OTO(r_n,p_n)$.}
\vspace{0.1cm}\\
In order to build $\hat{m}$, we remark that, for all $m$ in $\M_n$,
\begin{equation}\label{etoile}
\|s-\hat{s}_m\|^2=\|\hat{s}_m\|^2-2P\hat{s}_m+\|s\|^2=P_nQ(\hat{s}_m)+2\nu_n(\hat{s}_m)+\|s\|^2,
\end{equation}
where $\nu_n=P_n-P$ is the centered empirical process. An oracle minimizes $\|s-\hat{s}_m\|^2$ and thus $P_nQ(\hat{s}_m)+2\nu_n(\hat{s}_m)$. As we want to imitate the oracle, we will design a map $\pen :\M_n\rightarrow \R^+$ and choose
\begin{equation}\label{plse}
\hat{m}\in \arg\min_{m\in\M_n}P_nQ(\hat{s}_m)+\pen(m),\;\tilde{s}=\hat{s}_{\hat{m}}.
\end{equation}
It is clear that the ideal penalty is $\pen_{id}(m)=2\nu_n(\hat{s}_m)$. For all $m$ in $\M_n$, for all orthonormal basis $(\psi_{\lambda})_{\lambda\in m}$, $\hat{s}_m=\sum_{\lambda\in m} (P_n\psi_{\lambda})\psi_{\lambda}$ and $s_m=\sum_{\lambda\in m} (P\psi_{\lambda})\psi_{\lambda}$, thus 
$$\nu_n(\hat{s}_m-s_m)=\nu_n\left(\sum_{\lambda\in m} (\nu_n\psi_{\lambda})\psi_{\lambda}\right)=\sum_{\lambda\in m} (\nu_n\psi_{\lambda})^2=\|\hat{s}_m-s_m\|^2.$$
Let us define, for all $m$ in $\M_n$
$$p(m)=\nu_n(\hat{s}_m-s_m)=\|\hat{s}_m-s_m\|^2.$$
From (\ref{etoile}), for all $m$ in $\M_n$,
\begin{eqnarray*}
\left\|s-\tilde{s}\right\|^2&=&\|\tilde{s}\|^2-2P\tilde{s}+\left\|s\right\|^2=\|\tilde{s}\|^2-2P_n\tilde{s}+2\nu_n\tilde{s}+\left\|s\right\|^2\nonumber\\
&\leq&P_nQ(\hat{s}_{m})+\pen(m)+\left(2\nu_n(\tilde{s})-\pen(\hat{m})\right)+\left\|s\right\|^2\nonumber\\
&=&\left\|s-\hat{s}_{m}\right\|^2+\left(\pen(m)-2\nu_n(\hat{s}_{m})\right)+\left(2\nu_n(\tilde{s})-\pen(\hat{m})\right)
\end{eqnarray*}
Hence, for all $m$ in $\M_n$,
\begin{equation}\label{decRis}
\left\|s-\tilde{s}\right\|^2\leq\left\|s-\hat{s}_{m}\right\|^2+\left(\pen(m)-2p(m)\right)+\left(2p(\hat{m})-\pen(\hat{m})\right)+2\nu_n(s_{\hat{m}}-s_{m}).
\end{equation}
Let us define, for all $c_1,c_2>0$, the function 
\begin{equation}\label{fcc}
f_{c_1,c_2}:\R^+\rightarrow \R^+,\;x\mapsto \left\{
\begin{array}{lcc}
\frac{1+c_1x}{1-c_2x}-1&\textrm{if}&x<1/c_2\\
+\infty&\textrm{if}&x\geq 1/c_2
\end{array}\right. .
\end{equation}
It comes from inequality (\ref{decRis}) that $\tilde{s}$ satisfies an oracle inequality $OTO(f_{2,2}(\epsilon_n),p_n)$ as soon as, with probability larger than $1-p_n$
\begin{eqnarray}
\forall m\in \M_n\; \frac{\left|2p(m)-\pen(m)\right|}{\left\|s-\hat{s}_{m}\right\|^2}\leq \epsilon_n\;\textrm{and}\label{Main}\\
\forall (m,m')\in \M_n^2,\;\frac{2\nu_n(s_{m'}-s_{m})}{\left\|s-\hat{s}_{m'}\right\|^2+\left\|s-\hat{s}_{m}\right\|^2}\leq \epsilon_n.\label{Couple}
\end{eqnarray}
Inequality (\ref{Couple}) does not depend on our choice of penalty, we will check that it can easily be satisfied in classical collections of models. In order to obtain inequality (\ref{Main}), we use two methods, defined in $M$-estimation, but studied only on some regression frameworks.

\subsubsection{The slope heuristic}\label{ISH}
The first one is refered as the "slope heuristic". 
The idea has been introduced by Birg\'e $\&$ Massart \cite{BM07} in the Gaussian regression framework and developed in a general algorithm by Arlot $\&$ Massart \cite{AM08}. This heuristic states that there exist a sequence $(\Delta_m)_{m\in \M_n}$ and a constant $K_{\min}$ satisfying the following properties,
\begin{enumerate}
\item when $\pen(m)<K_{\min}\Delta_m$, then $\Delta_{\hat{m}}$ is too large, typically $\Delta_{\hat{m}}\geq C\max_{m\in\M_n}\Delta_m$,
\item when $\pen(m)\simeq(K_{\min}+\delta)\Delta_m$ for some $\delta>0$, then $\Delta_{\hat{m}}$ is much smaller,
\item when  $\pen(m)\simeq2K_{\min}\Delta_m$, the selected estimator is optimal.
\end{enumerate}
Thanks to the third point, when $\Delta_m$ and $K_{\min}$ are known, this heuristic says that the penalty $\pen(m)=2K_{\min}\Delta_m$ selects an optimal estimator. When $\Delta_m$ only is known, the first and the second point can be used to calibrate $K_{\min}$ in practice, as shown by the following algorithm (see Arlot $\&$ Massart \cite{AM08}):

\vspace{0.2cm}

\noindent{\bf Slope algorithm}\\
For all $K>0$, compute the selected model $\hat{m}(K)$ given by (\ref{plse}) with the penalty $\pen(m)=K\Delta_m$ and the associated complexity $\Delta_{\hat{m}(K)}$.\\
Find the constant $K_{\min}$ such that $\Delta_{\hat{m}(K)}$ is large when $K<K_{\min}$, and "much smaller" when $K>K_{\min}$.\\
Take the final $\hat{m}=\hat{m}(2K_{\min})$.

\vspace{0.2cm}

\noindent We will justify the slope heuristic in the density estimation framework for $\Delta_m=\E(\|s_m-\hat{s}_m\|^2)=D_m/n$ and $K_{\min}=1$. In general, $D_m$ is unknown and has to be estimated, we propose a resampling estimator and prove that it can be used without extra assumptions to obtain optimal results.

\subsubsection{Resampling penalties}
Data-driven penalties have already been used in density estimation in particular cross-validation methods as in Stone \cite{St74}, Rudemo \cite{Ru82} or Celisse \cite{Ce08}. We are interested here in the resampling penalties introduced  by Arlot \cite{Ar08}. Let $(W_1,...,W_n)$ be a resampling scheme, i.e. a vector of random variables independent of $X,X_1,...,X_n$ and exchangeable, that is, for all permutations $\tau$ of $(1,...,n)$,
$$(W_1,...,W_n)\;\textrm{has the same law as}\; (W_{\tau(1)},...,W_{\tau(n)}).$$
Hereafter, we denote by $\bar{W}_n=\sum_{i=1}^nW_i/n$ and by $E^W$ and $\mathcal{L}^W$ respectively the expectation and the law conditionally to the data $X,X_1,...,X_n$. Let $P_n^W=\sum_{i=1}^nW_i\delta_{X_i}/n$, $\nu_n^W=P_n^W-\bar{W}_nP_n$ be the resampled empirical processes. Arlot's procedure is based on the resampling heurististic of Efron (see Efron \cite{Ef79}), which states that the law of a functional $F(P,P_n)$ is close to its resampled counterpart, that is the conditional law $\mathcal{L}^W(C_WF(\bar{W}_nP_n,P^W_n))$. $C_W$ is a renormalizing constant that depends only on the resampling scheme and on $F$. Following this heuristic, Arlot defines as a penalty  the resampling estimate of the ideal penalty $2D_m/n$, that is 
\begin{equation}\label{Penm}
\pen(m)=2C_W\E^W(\nu_n^W(\hat{s}^W_{m})),
\end{equation}
where $\hat{s}^W_{m}$ minimizes $P_n^WQ(t)$ over $S_m$. We prove concentration inequalities for $\pen(m)$ and deduce that $\pen(m)$ provides an optimal procedure.

\vspace{0.2cm}

\noindent The paper is organized as follows. In Section \ref{Ch2S2}, we state our main results, we prove the efficiency of the slope algorithm and the resampling penalties.\\
In Section \ref{Ch2S3}, we compute the rates of convergence in the oracle inequalities using classical collections of models. Section \ref{Ch2S4} is devoted to a short simulation study where we compare different methods in practice. The proofs are postponed to Section \ref{Ch2S5}. Section \ref{Ch2S6} is an Appendix where we add some probabilistic material, we prove a concentration inequality for $Z^2$, where $Z=\sup_{t\in B}\nu_n(t)$ and $B$ is symmetric. We deduce a simple concentration inequality for $U$-statistics of order 2 that extends a previous result by Houdr\'e $\&$ Reynaud-Bouret \cite{HRB03}.

\section{Main results}\label{Ch2S2}

Hereafter, we will denote by $c$, $C$, $K$, $\kappa$, $L,$ $\alpha$, with various subscripts some constants that may vary from line to line.

\subsection{Concentration of the ideal penalty}\label{Cip}
Take an orthonormal basis $(\psi_{\lambda})_{\lambda\in m}$ of $S_m$. Easy algebra leads to 
$$s_m=\sum_{\lambda\in m}(P\psi_{\lambda})\psi_{\lambda},\;\hat{s}_m=\sum_{\lambda\in m}(P_n\psi_{\lambda})\psi_{\lambda}, \;\textrm{thus}\; \|s_m-\hat{s}_m\|^2=\sum_{\lambda\in m}(\nu_n(\psi_{\lambda}))^2.$$
$\hat{s}_m$ is an unbiased estimator of $s_m$ and 
$$\pen_{id}(m)=2\nu_n(\hat{s}_m)=2\nu_n(\hat{s}_m-s_m)+2\nu_n(s_m)=2\|s_m-\hat{s}_m\|^2+2\nu_n(s_m).$$
For all $m,m'$ in $\M_n$, let 
\begin{equation}\label{decpen}
p(m)=\|s_m-\hat{s}_m\|^2=\sum_{\lambda\in m}(\nu_n(\psi_{\lambda}))^2,\;\delta(m,m')=2\nu_n(s_m-s_{m'}).
\end{equation}
From (\ref{decRis}), for all $m$ in $\M_n$,
\begin{equation}\label{or1}
\left\|s-\tilde{s}\right\|_2^2\leq \left\|s-\hat{s}_{m}\right\|_2^2+\left(\pen(m)-2p(m)\right)+\left(2p(\hat{m})-\pen(\hat{m})\right)+\delta(\hat{m},m).
\end{equation}
In this section, we are interested in the concentration of $p(m)$ around $\E(p(m))=D_m/n$. Let us first remark that, for all $m$ in $\M_n$, $p(m)$ is the supremum of the centered empirical process over the ellipsoid $B_m=\{t\in S_m,\;\|t\|\leq 1\}$. From Cauchy-Schwarz inequality, for all real numbers $(b_{\lambda})_{\lambda\in m}$, 
\begin{equation}\label{CSI}
\sum_{\lambda\in m}b_{\lambda}^2=\left(\sup_{\sum a_{\lambda}^2\leq 1}\sum_{\lambda\in m}a_{\lambda}b_{\lambda}\right)^2.
\end{equation}
We apply this inequality with $b_{\lambda}=\nu_n(\psi_{\lambda})$. We obtain, since the system $(\psi_{\lambda})_{\lambda\in m}$ is orthonormal,
$$\sum_{\lambda\in m}(\nu_n(\psi_{\lambda}))^2=\sup_{\sum a_{\lambda}^2\leq 1}\left(\sum_{\lambda\in m}a_{\lambda}\nu_n(\psi_{\lambda})\right)^2=\sup_{\sum a_{\lambda}^2\leq 1}\left(\nu_n\left(\sum_{\lambda\in m}a_{\lambda}\psi_{\lambda}\right)\right)^2=\sup_{t\in B_m}\left(\nu_n(t)\right)^2.$$
Hence, $p(m)$ is bounded by a Talagrand's concentration inequality (see Talagrand \cite{Ta96}). This inequality involves $D_m=n\E\left(\|\hat{s}_m-s_m\|^2\right)$ and the constants 
\begin{equation}\label{defbv}
e_m=\frac1n\sup_{t\in B_m}\left\|t\right\|^2_{\infty}\;\textrm{and}\;v_m^2=\sup_{t\in B_m}\Var(t(X)).
\end{equation}
More precisely, the following proposition holds:
\begin{prop}\label{concbasic}
Let $X,X_1,...,X_n$ be iid random variables with common density $s$ with respect to a probability measure $\mu$. Assume that $s$ belongs to $L^2(\mu)$ and let $S_m$ be a linear subspace in $L^2(\mu)$. Let $s_m$ and $\hat{s}_m$ be respectively the orthogonal projection and the projection estimator of $s$ onto $S_m$. Let $p(m)=\|s_m-\hat{s}_m\|^2$, $D_m=n\E(p(m))$ and let $v_m$, $e_m$ be the constants defined in (\ref{defbv}). Then, for all $x>0$, 
\begin{equation}\label{minpen1}
\p\left(p(m)-\frac{D_m}n>\frac{D_m^{3/4}(e_mx^2)^{1/4}+0.7\sqrt{D_mv_m^2x}+0.15v_m^2x+e_m x^2}n\right)\leq e^{-x/20}
\end{equation}
\begin{equation}\label{minpen2}
\p\left(\frac{D_m}n-p(m)>\frac{1.8D_m^{3/4}(e_mx^2)^{1/4}+1.71\sqrt{D_mv_m^2x}+4.06e_m x^2}n\right)\leq 2.8e^{-x/20}
\end{equation}
\end{prop}
{\bf Comments :} From (\ref{or1}), for all $m$ in $\M_n$,
\begin{eqnarray}
\left\|s-\tilde{s}\right\|_2^2&\leq& \left\|s-\hat{s}_{m}\right\|_2^2+\left(\pen(m)-2\frac{D_m}n\right)+2\left(\frac{D_m}n-p(m)\right)\nonumber\\
&&+2\left(p(\hat{m})-\frac{D_{\hat{m}}}n\right)+\left(2\frac{D_{\hat{m}}}n-\pen(\hat{m})\right)+\delta(\hat{m},m).\label{o2}
\end{eqnarray}
It appears from (\ref{o2}) that we can obtain oracle inequalities with a penalty of order $2D_m/n$ if, uniformly over $m,m'$ in $\M_n$,
$$p(m)-\frac{D_m}n<<\|s-\hat{s}_m\|^2\;\textrm{and}\;\delta(m',m)<<\|s-\hat{s}_m\|^2+\|s-\hat{s}_{m'}\|^2.$$
Proposition \ref{concbasic} proves that the first part holds with large probability for all $m$ in $\M_n$ such that $e_m\vee v_m^2<<n\E(\|s-\hat{s}_m\|^2)$. Actually, the other part also holds under the same kind of assumption.

\subsection{Main assumptions}\label{MainAss}
For all $m$, $m'$ in $\M_n$, let $D_m=n\E\left(\|s_m-\hat{s}_m\|^2\right)$,
$$\frac{R_m}n=\E\left(\|s-\hat{s}_m\|^2\right)=\|s-s_m\|^2+\frac{D_m}n,$$
$$v_{m,m'}^2=\sup_{t\in S_m+S_{m'},\|t\|\leq 1}\Var(t(X)), e_{m,m'}=\frac1n\sup_{t\in S_m+S_{m'},\|t\|\leq 1}\left\|t\right\|^2_{\infty}.$$
For all $k\in \N$, let $\M_n^k=\{m\in\M_n,\; R_m\in [k,k+1)\}$. For all $n$ in $\N$, for all $k>0$, $k'>0$ and $\gamma\geq 0$, let $[k]$ be the integer part of $k$ and let
\begin{equation}\label{marge}
l_{n,\gamma}(k,k')=\ln(1+{\rm Card}(\M_n^{[k]}))+\ln(1+{\rm Card}(\M_n^{[k']}))+\ln((k+1)(k'+1))+(\ln n)^{\gamma}
\end{equation}

\vspace{0.2cm}

\noindent{\bf Assumption [V]}: {\it{There exist $\gamma>1$ and a sequence $(\epsilon_n)_{n\in\N}$, with $\epsilon_n\rightarrow 0$ such that, for all $n$ in $\N$,}}
\begin{equation*}
\sup_{(k,k')\in(\N^*)^2}\sup_{(m,m')\in\M^k_n\times \M_n^{k'}}\left\{\left(\left(\frac{v_{m,m'}^2}{R_m\vee R_{m'}}\right)^2\vee \frac{e_{m,m'}}{R_m\vee R_{m'}}\right)l^2_{n,\gamma}(k,k')\right\}\leq\epsilon^4_n.
\end{equation*}
\vspace{0.2cm}\\
{\bf [BR]} {\it There exist two sequences $(h^*_n)_{n\in\N^*}$ and $(h^o_n)_{n\in\N^*}$ with $(h^o_n\vee h^*_n)\rightarrow 0$ as $n\rightarrow \infty$ such that, for all $n$ in $\N^*$, for all $m_o\in\arg\min_{m\in\M_n}R_m$ and all $m^*\in\arg\max_{m\in\M_n}D_m$,}
$$ \frac{R_{m_o}}{D_{m^*}}\leq h^o_n,\; \frac{n\|s-s_{m^*}\|^2}{D_{m^*}}\leq h^*_n.$$
\noindent{\bf Comments:} 
\begin{itemize}
\item Assumption {\bf [V]} ensures that the fluctuations of the ideal penalty are uniformly small compared to the risk of the estimator $\hat{s}_m$. Note that for all $k,k'$, $l_{n,\gamma}(k,k')\geq (\ln n)^{\gamma}$, thus, Assumption {\bf [V]} holds only in typical non parametric situations where $R_n=\inf_{m\in\M_n}R_{m}\rightarrow \infty$ as $n\rightarrow \infty$.
\item The slope heuristic states that the complexity $\Delta_{\hat{m}}$ of the selected estimator is too large when the penalty term is too small. A minimal assumption for this heuristic to hold with $\Delta_m=D_m$ would be that there exists a sequence $(\theta_n)_{n\in\N^*}$ with $\theta_n\rightarrow 0$ as $n\rightarrow \infty$ such that, for all $n$ in $\N^*$, for all $m_o\in\arg\min_{m\in\M_n}\E\left(\|s-\hat{s}_m\|^2\right)$ and all $m^*\in\arg\max_{m\in\M_n}\E\left(\|s_m-\hat{s}_m\|^2\right)$,
\begin{equation*}
D_{m_o}\leq \theta_nD_{m^*}.
\end{equation*}
Assumption {\bf [BR]} is slightly stronger but will always hold in the examples (see Section \ref{Ch2S3}).
\end{itemize}
In order to have an idea of the rates $R_n, \epsilon_n$, $h^*_n, h^o_n$ and $\theta_n$, let us briefly consider the very simple following example:
\vspace{0.2cm}\\
{\bf Example HR:} We assume that $s$ is supported in $[0,1]$ and that $(S_m)_{m\in\M_n}$ is the collection of the regular histograms on $[0,1]$, with $d_m=1,...,n$ pieces. We will see in Section \ref{THC} that $D_m\sim d_m$ asymptotically, hence $D_{m^*}\simeq n$. Moreover, we assume that $s$ is H\"olderian and not constant so that there exist positive constants $c_l,c_u, \alpha_l, \alpha_u$ such that, for all $m$ in $\M_n$, see for example Arlot \cite{Ar08},
$$c_ld_m^{-\alpha_l}\leq \|s-s_m\|^2\leq c_ud_m^{-\alpha_u}.$$
In Section \ref{THC}, we prove that this assumption implies {\bf [V]} with $\epsilon_n\leq C\ln(n)n^{-1/(8\alpha_l+4)}$.\\
Moreover, there exists a constant $C>0$ such that $R_{m_o}\leq \inf_{m\in\M_n}\left(c_und_{m}^{-\alpha_u}+d_m\right)\leq Cn^{-1/(2\alpha_u+1)}$, thus  $R_{m_o}/D_m^*\leq Cn^{1/(2\alpha_u+1)-1}=Cn^{-2\alpha_u/(2\alpha_u+1)}$. Since there exists $C>0$ such that $n\|s-s_{m^*}\|^2/D_{m^*}\leq Cd_{m^*}^{-\alpha_u}=Cn^{-\alpha_u}$, {\bf [BR]} holds with $h^o_n=Cn^{-2\alpha_u/(2\alpha_u+1)}$ and $h^*_n=Cn^{-\alpha_u}$.\\
Other examples can be found in Birg\'e $\&$ Massart \cite{BM97}, see also Section \ref{Ch2S3}.

\subsection{Results on the Slope Heuristic}\label{RSH}
Let us now turn to the slope heuristic presented in Section \ref{ISH}.

\begin{theo}\label{P2} (Minimal penalty)
Let $\M_n$ be a collection of models satisfying {\bf [V]} and {\bf [BR]} and let $\epsilon^*_n=\epsilon_n\vee h_n^*$.\\
Assume that there exists $0<\delta_n<1$ such that $0\leq \pen(m)\leq (1-\delta_n)D_m/n$. Let $\hat{m},\tilde{s}$ be the random variables defined in (\ref{plse}) and let
$$c_n=\frac{\delta_n-28\epsilon^*_n}{1+16\epsilon_n}.$$
There exists a constant $C>0$ such that,
\begin{equation}\label{slope1}
\p\left(D_{\hat{m}}\geq c_nD_{m^*},\;\|s-\tilde{s}\|^2\geq \frac{c_n}{5h_n^o}\inf_{m\in\M_n}\|s-\hat{s}_m\|^2\right)\geq1-Ce^{-\frac12(\ln n)^{\gamma}}.
\end{equation}
\end{theo}
{\bf Comments:} Assume that $\pen(m)\leq (1-\delta)D_m/n$, then, inequality (\ref{slope1}) proves that an oracle inequality can not be obtained since $c_n/h_n^o\rightarrow \infty$. Moreover, $D_{\hat{m}}\geq c D_{m^*}$ is as large as possible. This proves point 1 of the slope heuristic.

\begin{theo}\label{P3}
Let $\M_n$ be a collection of models satisfying Assumption {\bf [V]}. Assume that there exist $\delta^+\geq \delta_->-1$ and $0\leq p'<1$ such that, with probability at least $1-p'$,
$$2\frac{D_m}n+\delta_-\frac{R_m}n\leq \pen(m)\leq2\frac{D_m}n+\delta^+\frac{R_m}n.$$
Let $\hat{m},\tilde{s}$ be the random variables defined in (\ref{plse}) and let
$$C_n(\delta_-,\delta^+)=\left(\frac{1+\delta_--46\epsilon_n}{1+\delta^++26\epsilon_n}\vee 0\right)^{-1}.$$ 
There exists a constant $C>0$ such that, with probability larger than $1-p'-Ce^{-\frac12(\ln n)^{\gamma}}$,
\begin{equation}\label{OISH}
D_{\hat{m}}\leq C_n(\delta_-,\delta^+)R_{m_o}
,\;\|s-\tilde{s}\|^2\leq C_n(\delta_-,\delta^+)\inf_{m\in\M_n}\|s-\hat{s}_m\|^2.
\end{equation}
\end{theo}
{\bf Comments :} 
\begin{itemize}
\item Assume that $\pen(m)=KD_m/n$ with $K>1$, then inequality (\ref{OISH}) ensures that $D_{\hat{m}}\leq C_n(K,K)R_{m_o}$. Hence, $D_{\hat{m}}$ jumps from $D_{m^*}$ (Theorem \ref{P2}) to $R_{m_o}$ (\ref{OISH}) when $\pen(m)$ is around $D_m/n$, which is much smaller thanks to Assumption {\bf[BR]}. This proves point 2 of the slope heuristic. 
\item Point 3 of this heuristic comes from inequality (\ref{OISH}) applied with small $\delta_-$ and $\delta^+$. The rate of convergence of the leading constant to $1$ is then given by the supremum between $\delta_-$, $\delta^+$ and $\epsilon_n$. 
\item The condition on the penalty has the same form as the one given in Arlot $\&$ Massart \cite{AM08}. It comes from the fact that we do not know $D_m/n$ in many cases, therefore, it has to be estimated. We propose two alternatives to solve this issue. In Section \ref{Rp}, we give a resampling estimator of $D_m$. It can be used for all collection of models satisfying {\bf [V]} and its error of approximation is upper bounded by $\epsilon_n R_m/n$. Thus Theorem \ref{P3} holds with $(\delta_-\vee \delta^+)\leq C\epsilon_n$. In Section \ref{THC}, we will also see that, in regular models, we can use $d_m$ instead of $D_m$ and the error is upper bounded by $C R_m/R_{m_o}$, thus Theorem \ref{P3} holds with $(\delta_-\vee \delta^+)\leq C/R_{m_o}<<\epsilon_n$, $p'=0$. In both cases, we deduce from Theorem \ref{P3} that the estimator $\tilde{s}$ given by the slope algorithm achieves an optimal oracle inequality $OTO(\kappa\epsilon_n, Ce^{-\frac12(\ln n)^{\gamma}})$. In Example {\bf HR}, for example, we obtain $\epsilon_n=Cn^{-1/(8\alpha_l+4)}\ln n$.
\end{itemize}

\subsection{Resampling penalties}\label{Rp}
Optimal model selection is possible in density estimation provided that we have a sharp estimation of $D_m=n\E\left(\sup_{t\in B_m}(\nu_n(t))^2\right)$. We propose an estimator of this quantity based on the resampling heuristic. The model selection algorithm that we deduce is the same as the resampling penalization procedure introduced by Arlot \cite{Ar08}. Let $F$ be a fixed functional. Efron's heuristic states that the law $\mathcal{L}(F(\nu_n))$ is close to the conditional law $\mathcal{L}^W(C_WF(\nu_n^W))$, where $C_W$ is a normalizing constant depending only on the resampling scheme and the functional $F$. Let $P_n^W=\sum_{i=1}^nW_i\delta_{X_i}/n$ and $\nu_n^W=P_n^W-\bar{W}_nP_n$. The resampling estimator of $D_m$ is $D_m^W=nC_W^2\E^W\left(\sup_{t\in B_m}(\nu_n^W(t))^2\right)$ and the resampling penalty associated is $\pen(m)=2D_m^W/n$. Actually, the following result describes the concentration of $D_m^W$ around its mean $D_m$ and around $np(m)$.
\begin{prop}\label{cboot}
Let $(W_1,...,W_n)$ be a resampling scheme, let $S_m$ be a linear space, $B_m=\left\{t\in S_m,\; \|t\|\leq 1\right\}$, $p(m)=\sup_{t\in B_m}(\nu_n(t))^2$, $D_m=n\E\left(p(m)\right)$ and let $D_m^W$ be the resampling estimator of $D_m$ based on $(W_1,...,W_n)$, that is $D_m^W=nC_W^2\E^W\left(\sup_{t\in B_m}(\nu_n^W(t))^2\right)$, where $v_W^2=\Var(W_1-\bar{W}_n)$ and $C_W^2=(v_W^2)^{-1}$.\\
Then, for all $m$ in $\M_n$, $\E(D_m^W)=D_m$. Moreover, let $e_m$, $v_m$ be the quantities defined in (\ref{defbv}). For all $x>0$, on an event of probability larger than $1-7.8e^{-x}$,
\begin{eqnarray}
D_m^W-D_m&\leq&\sqrt{8e_mD_mx}+e_m\left(\frac{4x}3+\frac{ (40.3x)^2}{n-1}\right)\nonumber\\
&&+\frac{9D_m^{3/4}(e_m x^2)^{1/4}+7.61\sqrt{v_m^2D_mx}}{n-1}.\label{concboot1}\\
D_m^W-D_m&\geq&-\sqrt{8e_mD_mx}-e_m\left(\frac{4x}3+\frac{(19.1x)^2}{n-1}\right)\nonumber\\
&&-\frac{5.31D_m^{3/4}(e_m x^2)^{1/4}+3\sqrt{v_m^2D_mx}+3v_m^2x}{n-1}.\label{concboot2}
\end{eqnarray}
For all $x>0$,
\begin{equation}\label{concboot3}
\p\left( p(m)-\frac{D_m^W}n>\frac{5.31D_m^{3/4}(e_m x^2)^{1/4}+3\sqrt{v_m^2D_mx}+3v_m^2x+e_m (19.1x)^2}{n-1}\right)\leq2e^{-x}
\end{equation}
\begin{equation}\label{concboot4}
\p\left( \frac{D_m^W}n-p(m)\leq\frac{9D_m^{3/4}(e_m x^2)^{1/4}+7.61\sqrt{v_m^2D_mx}+e_m (40.3x)^2}{n-1}\right)\leq3.8e^{-x}.
\end{equation}
\end{prop}
{\bf Remark }\\
The concentration of the resampling estimator involves the same quantities as the concentration of $p(m)$, thus, it can be used to estimate the ideal penalty in the slope heuristic's algorithm presented in the previous section without extra assumptions on the collection $\M_n$. Proposition \ref{cboot} and Theorem \ref{P3} prove that this resampling penalty leads to an efficient model selection procedure. However, we do not need to use the slope heuristic in our framework to obtain an optimal model selection procedure as shown by the following theorem.

\begin{theo}\label{OracleB1}
Let $X_1,...,X_n$ be i.i.d random variables with common density $s$. Let $\M_n$ be a collection of models satisfying Assumption {\bf [V]}. Let $W_1,...,W_n$ be a resampling scheme, let $\bar{W}_n=\sum_{i=1}^nW_i/n$, $v_W^2=\Var(W_1-\bar{W}_n)$ and $C_W=2(v_W^2)^{-1}$. Let $\tilde{s}$ be the penalized least-squares estimator defined in (\ref{plse}) with 
$$\pen(m)=C_W\E^W\left(\sup_{t\in B_m}(\nu_n^W(t))^2\right).$$
Then, there exists a constant $C>0$ such that 
\begin{equation}\label{Or1}
\p\left(\|s-\tilde{s}\|^2\leq(1+100\epsilon_n)\inf_{m\in\M_n}\|s-\hat{s}_m\|^2\right)\geq 1-C e^{-\frac12(\ln n)^{\gamma}}.
\end{equation}
\end{theo}
{\bf Comments :} The main advantage of this results is that the penalty term is always totally computable. Unlike the penalties derived from the slope heuristic, it does not depend on an arbitrary choice of a constant $K_{\min}$ made by the observer, that may be hard to detect in practice (see the paper of Alot $\&$ Massart \cite{AM08} for an extensive discussion on this important issue). However, $C_W$ is only optimal asymptotically. It is sometimes useful to overpenalize a little in order to improve the non-asymptotic performances of our procedures (see Massart \cite{Ma07}) and the slope heuristic can be used to do it in an optimal way (see our short simulation study in Section \ref{Ch2S4}).

\subsection{A remarks on the "regularization phenomenon"}
The regularization of the bootstrap phenomenon (see Arlot  \cite{Ar07, Ar08} and the references therein) states that the resampling estimator $C_W\E^W(F(\nu_n^W))$ of a functional $F(\nu_n)$ concentrates around its mean better than $F(\nu_n)$. This phenomenon can be justified with our previous results for our functional $F$. Recall that we have proven in Proposition \ref{concbasic} that, for all $x>0$, with probability larger than $1-3.8e^{-x/20}$,
\begin{equation*}
\left|p(m)-\frac{D_m}n\right|\leq\frac{1.8D_m^{3/4}(e_mx^2)^{1/4}+1.5\sqrt{D_mv_m^2x}+0.2v_m^2x+4.1e_m x^2}n.
\end{equation*}
In Example {\bf HR}, we have the following upper bounds
$$D_m\leq d_m,\;e_m\leq \frac{d_m}n,\; v_m^2\leq c\|s\|\sqrt{d_m}.$$
Thus, there exists a constant $C$ such that, for all $x>0$,
\begin{equation}\label{empi}
\p\left(\left|n p(m)-D_m\right|>Cd_m\left(\sqrt{\frac x{\sqrt{n}}}+\left(\frac{x}{\sqrt{n}}\right)^2\right)\right)\leq 3.8e^{-x/20}.
\end{equation}
On the other hand, it comes from Inequalities (\ref{concboot1}) and (\ref{concboot2}), that, for all $x>0$, on an event of probability larger than $1-7.8e^{-x/20}$,
\begin{eqnarray*}
\left|D_m^W-D_m\right|&\leq&\sqrt{0.4e_mD_mx}+e_m\left(\frac{x}{15}+\frac{4.1 x^2}{n-1}\right)\\
&&+\frac{1.8D_m^{3/4}(e_m x^2)^{1/4}+1.45\sqrt{v_m^2D_mx}+0.2v_m^2x}{n-1}.
\end{eqnarray*}
Thus, there exists a constant $C$ such that, for all $x>0$,
\begin{equation*}
\p\left(\left|D_m^W-D_m\right|> Cd_m\left(\sqrt{\frac xn}+\left(\frac xn\right)^2\right)\right)\leq 7.8e^{-x/20}.
\end{equation*}
The concentration of $D_m^W$ is then much better than the one of $np(m)$. This implies that $D_m^W$ is an estimator of $D_m$ rather than an estimator of $n p(m)$. Thus, the resampling penalty can be used when $D_m/n$ is a good penalty for example, under {\bf[V]}. When $D_m/n$ is known to underpenalize (see the examples in Barron, Birg\'e $\&$ Massart \cite{BBM99}), there is no chance that $D_m^W/n$ can work.

\section{Rates of convergence for classical examples}\label{Ch2S3}

The aim of this section is to show that {\bf [V]} can be derived from a more classical hypothesis in two classical collections of models: the histograms and Fourier spaces. We derive the rates $\epsilon_n$ under this new hypothesis.

\subsection{Assumption on the risk of the oracle}
As mentioned in Section \ref{MainAss}, Assumption {\bf [V]} can only hold if there exists $\gamma>1$ such that $R_n(\ln n)^{-\gamma}\rightarrow \infty$ as $n\rightarrow \infty$, where $R_n=\inf_{m\in\M_n}R_m$. In our example, we will make the following Assumption that ensures that this condition is always satisfied.
\vspace{0.2cm}\\
{\bf[BR]} {\it (Bounds on the Risk) There exist constants $C_u>0$, $\alpha_u>0$, $\gamma>1$, and a sequence $(\theta_n)_{n\in\N}$ with $\theta_n\rightarrow \infty$ as $n\rightarrow \infty$ such that, for all $n$ in $\N^*$, for all $m$ in $\M_n$}
$$\theta_n^2(\ln n)^{2\gamma}\leq R_n\leq R_m\leq C_un^{\alpha_u}.$$
{\bf Comments:} Assumption {\bf[BR]} holds with $\theta_n=Cn^{\alpha}$ for the collection of regular histograms of example {\bf HR}, provided that $s$ is an H\"olderian, non constant and compactly supported function (see for example Arlot \cite{Ar07}). It is also a classical result of minimax theory that there exist functions in Sobolev spaces satisfying this kind of Assumption when $\M_n$ is the collection of Fourier spaces that we will introduce below.
\vspace{0.2 cm}\\
We want to check that these collections satisfy Assumption {\bf [V]}, i.e. that there exists $\gamma>1$ such that 
$$\sup_{(k,k')\in(\N^*)^2}\sup_{(m,m')\in\M^k_n\times \M_n^{k'}}\left\{\left(\left(\frac{v_{m,m'}^2}{R_m\vee R_{m'}}\right)^2\vee \frac{e_{m,m'}}{R_m\vee R_{m'}}\right)l^2_{n,\gamma}(k,k')\right\}\leq\epsilon^4_n.$$
For all $m\in\M_n$, $R_m\leq C_un^{\alpha_u}$, thus for all $k>C_un^{\alpha_u}$, ${\rm Card}(\M_n^k)=0$. In particular, we can assume in the previous supremum that $k\leq C_un^{\alpha_u}$ and $k'\leq C_un^{\alpha_u}$. Hence, there exists a constant $\kappa>0$ such that $\ln[(1+k)(1+k')]\leq \kappa\ln n$. We also add the following assumption that ensures that there exists a constant $\kappa>0$ such that, for all $k\in\N$, $\ln(1+{Card}(\M_n^k))\leq \kappa\ln n$.
\vspace{0.2cm}\\
{\bf [PC]} {\it (Polynomial collection) There exist constants $c_{\M}\geq0$, $\alpha_{\M}\geq 0$, such that, for all $n$ in $\N$,}
$${\rm Card}(\M_n)\leq c_{\M}n^{\alpha_{\M}}.$$
Under Assumptions {\bf[BR]} and {\bf [PC]}, there exists a constant $\kappa>0$ such that, for all $\gamma>1$ and $n\geq 3$,
\begin{eqnarray*}
&&\sup_{(k,k')\in(\N^*)^2}\sup_{(m,m')\in\M^k_n\times \M_n^{k'}}\left\{\left(\left(\frac{v_{m,m'}^2}{R_m\vee R_{m'}}\right)^2\vee \frac{e_{m,m'}}{R_m\vee R_{m'}}\right)l^2_{n,\gamma}(k,k')\right\}\\
&&\leq\sup_{(m,m')\in(\M_n)^2}\left\{\left(\frac{v_{m,m'}^2}{R_m\vee R_{m'}}\right)^2\vee \frac{e_{m,m'}}{R_m\vee R_{m'}}\right\}\kappa(\ln n)^{2\gamma}.
\end{eqnarray*}

\subsection{The histogram case}\label{THC}
Let $(\x,\mathcal{X})$ be a measurable space. Let $(P_m)_{m\in \M_n}$ be a collection of measurable partitions $P_m=(I_{\lambda})_{\lambda\in m}$ of subsets of $\x$ such that, for all $m\in\M_n$, for all $\lambda\in m$, $0<\mu(I_{\lambda})<\infty$. Let $m$ in $\M_n$, the set $S_m$ of histograms associated to $P_m$ is the set of functions which are constant on each $I_{\lambda}$, $\lambda\in m$. $S_m$ is a linear space. Setting, for all $\lambda\in m$, $\psi_{\lambda}=(\sqrt{\mu(I_{\lambda})})^{-1}1_{I_{\lambda}}$, the functions $(\psi_{\lambda})_{\lambda\in m}$ form an orthonormal basis of $S_m$.\\
Let us recall that, for all $m$ in $\M_n$,
\begin{equation}\label{Dhisto}
D_m=\sum_{\lambda\in m}\Var(\psi_{\lambda}(X))=\sum_{\lambda\in m}P(\psi_{\lambda}^2)-(P\psi_{\lambda})^2=\sum_{\lambda\in m}\frac{P(X\in I_{\lambda})}{\mu(I_{\lambda})}-\|s_m\|^2.
\end{equation}
Moreover, from Cauchy-Schwarz inequality, for all $x$ in $\x$, for all $m$, $m'$ in $\M_n$
\begin{equation}\label{ehisto}
\sup_{t\in B_{m,m'}}t^2(x)\leq \sum_{\lambda\in m\cup m'}\psi_{\lambda}^2(x),\;\textrm{thus}\;e_{m,m'}=\frac 1n\sup_{\lambda\in m\cup m'}\frac{1}{\mu(I_{\lambda})}.
\end{equation}
Finally, it is easy to check that, for all $m$ ,$m'$ in $\M_n$
\begin{equation}\label{vhisto}
v_{m,m'}^2= \sup_{\lambda\in m\cup m'}\Var(\psi_{\lambda}(X))=\sup_{\lambda\in m\cup m'}\frac{P(X\in I_{\lambda})(1-P(X\in I_{\lambda}))}{\mu(I_{\lambda})}.
\end{equation}
We will consider two particular types of histograms.\\
{\bf Example 1 [Reg] : $\mu$-regular histograms.}\\ 
{\it For all $m$ in $\M_n$, $P_m$ is a partition of $\x$ and there exist a family $(d_m)_{m\in\M_n}$ bounded by $n$ and two constants $c_{rh}$, $C_{rh}$ such that, for all $m$ in $\M_n$, for all $\lambda\in \M_n$,}
$$\frac{c_{rh}}{d_m}\leq \mu(I_{\lambda})\leq \frac{C_{rh}}{d_m}.$$
The typical example here is the collection described in Example {\bf HR}.\\
\vspace{0.2cm}\\
{\bf Example 2 [Ada]: Adapted histograms.} \\
{\it There exist positive constants $c_r$, $C_{ah}$ such that, for all $m$ in $\M_n$, for all $\lambda\in \M_n$, $\mu(I_{\lambda})\geq c_{r}n^{-1}$ and}
$$\frac{P(X\in I_{\lambda})}{\mu(I_{\lambda})}\leq C_{ah}.$$
{\bf [Ada]} is typically satisfied when $s$ is bounded on $\x$. Remark that the models satisfying {\bf [Ada]} have finite dimension $d_m\leq Cn$ since 
$$1\geq \sum_{\lambda\in m}P(X\in I_{\lambda})\geq C_{ah}\sum_{\lambda\in m}\mu(I_{\lambda})\geq C_{ah}c_rd_m n^{-1}.$$
\vspace{0.2cm}\\
{\bf The example [Reg]}.\\
It comes from equations (\ref{Dhisto}, \ref{ehisto}, \ref{vhisto}) and Assumption {\bf [Reg]} that
$$C^{-1}_{rh}d_m-\|s_m\|^2\leq D_m\leq c^{-1}_{rh}d_m-\|s_m\|^2.$$
$$e_{m,m'}\leq c_{rh}^{-1}\frac{d_m\vee d_{m'}}n,\; v_{m,m'}^2\leq \sup_{t\in B_{m,m'}}\left\|t\right\|_{\infty}\|t\|\|s\|\leq c^{-1/2}_{rh}\|s\|\sqrt{d_m\vee d_{m'}}.$$
Thus 
$$\frac{e_{m,m'}}{R_m\vee R_{m'}}\leq C_{rh}c_{rh}^{-1}\frac{(R_m\vee R_{m'})+\|s\|^2}{n(R_m\vee R_{m'})}\leq Cn^{-1}.$$
If $D_m\vee D_{m'}\leq \theta_n^2(\ln n)^{2\gamma}$,
$$\frac{v_{m,m'}^2}{R_m\vee R_{m'}}\leq \sqrt{C_{rh}c_{rh}^{-1}}\frac{\sqrt{(D_m\vee D_{m'})+\|s\|^2}}{R_{m_o}}\leq \frac{C}{\theta_n(\ln n)^{\gamma}}.$$
If $D_m\vee D_{m'}\geq \theta_n^2(\ln n)^{2\gamma}$,
$$\frac{v_{m,m'}^2}{R_m\vee R_{m'}}\leq \sqrt{C_{rh}c_{rh}^{-1}}\frac{\sqrt{(D_m\vee D_{m'})+\|s\|^2}}{D_m\vee D_{m'}}\leq \frac{C}{\theta_n(\ln n)^{\gamma}}.$$
There exists $\kappa>0$ such that $\theta_n^2(\ln n)^{2\gamma}\leq \kappa n$ since for all $m$ in $\M_n$, $R_m\leq n\|s-s_m\|^2+c_{rh}^{-1}d_m\leq (\|s\|^2+c_{rh}^{-1}) n$.
Hence Assumption {\bf[V]} holds with $\gamma$ given in Assumption {\bf [BR]} and $\epsilon_n=C\theta_n^{-1/2}$.
\vspace{0.2cm}\\
{\bf The example [Ada]}.\\
It comes from inequalities (\ref{ehisto}), (\ref{vhisto}) and Assumption {\bf [Ada]} that, for all $m$ and $m'$ in $\M_n$
$$e_{m,m'}\leq c_r^{-1}\;\textrm{and}\; v_{m,m'}^2\leq C_{ah}.$$
Thus, there exists a constant $\kappa>0$ such that, for all $m$ an $m'$ in $\M_n$, 
$$\sup_{(m,m')\in(\M_n)^2}\left\{\left(\frac{v_{m,m'}^2}{R_m\vee R_{m'}}\right)^2\vee \frac{e_{m,m'}}{R_m\vee R_{m'}}\right\}\leq \frac{\kappa}{\theta_n^2(\ln n)^{2\gamma}}.$$
Therefore Assumption {\bf [V]} holds also with $\gamma$ given in Assumption {\bf [BR]} and $\epsilon_n=\kappa\theta^{-1/2}_n$.
\subsection{Fourier spaces}\label{Ch2S33}
In this section, we assume that $s$ is supported in $[0,1]$. We introduce the classical Fourier basis. Let $\psi_0:[0,1]\rightarrow \R,\; x\mapsto 1$ and, for all $k\in \N^*$, we define the functions 
$$\psi_{1,k}:[0,1]\rightarrow \R,\; x\mapsto \sqrt{2}\cos(2\pi k x),\;\psi_{2,k}:[0,1]\rightarrow \R,\; x\mapsto \sqrt{2}\sin(2\pi k x).$$
For all $j$ in $\N^*$, let
$$m_j=\{0\}\cup \{(i,k),\; i=1,2,\; k=1,...,j\}\; \textrm{and}\;\M_n=\{m_j, j=1,...,n\}.$$
For all $m$ in $\M_n$, let $S_m$ be the space spanned by the family $(\psi_{\lambda})_{\lambda\in m}$. $(\psi_{\lambda})_{\lambda\in m}$ is an orthonormal basis of $S_m$ and for all $j$ in $1,...,n$, $d_{m_j}=2j+1$.\\
Let $j$ in $1,...n$, for all $x$ in $[0,1]$,
$$\sum_{\lambda\in m_j}\psi_{\lambda}^2(x)=1+2\sum_{k=1}^j\cos^2(2\pi k x)+\sin^2(2\pi k x)=1+2j=d_{m_j}.$$
Hence, for all $m$ in $\M_n$,
\begin{equation}\label{Dfou}
D_m=P\left(\sum_{\lambda\in m_j}\psi_{\lambda}^2\right)-\|s_m\|^2=d_m-\|s_m\|^2.
\end{equation}
It is also clear that, for all $m$, $m'$ in $\M_n$,
\begin{equation}\label{evfou}
e_{m,m'}=\frac{d_m\vee d_{m'}}n,\;v_{m,m'}^2\leq \|s\|\sqrt{d_m\vee d_{m'}}.
\end{equation}
The collection of Fourier spaces of dimension $d_m\leq n$ satisfies Assumption {\bf [PC]}, and the quantities $D_m$ $e_{m,m'}$ and $v_{m,m'}^2$ satisfy the same inequalities as in the collection {\bf [Reg]}, therefore, {\bf[V]} comes also in this collection from {\bf [BR]}. We have obtained the following corollary of Theorem \ref{OracleB1}.

\begin{coro}
Let $\M_n$ be either a collection of histograms satisfying Assumptions {\bf[PC]-[Reg]} or {\bf[PC]-[Ada]} or the collection of Fourier spaces of dimension $d_m\leq n$. Assume that $s$ satisfies Assumption {\bf [BR]} for some $\gamma>1$ and $\theta_n\rightarrow \infty$. Then, there exist constants $\kappa>0$ and $C>0$ such that the estimator $\tilde{s}$ selected by a resampling penalty satisfies
$$\p\left(\|s-\tilde{s}\|^2\leq (1+\kappa\theta_n^{-1/2})\inf_{m\in\M_n}\|s-\hat{s}_m\|^2\right)\geq 1-Ce^{-\frac12(\ln n)^{\gamma}}.$$
\end{coro}
{\bf Comment:} Assumption {\bf [BR]} is hard to check in practice. We mentioned that it holds in Example {\bf HR} provided that $s$ is H\"olderian, non constant and compactly supported (see Arlot \cite{Ar08}). It is also classical to build functions satisfying {\bf [BR]} with the Fourier spaces in order to prove that the oracle reaches the minimax rate of convergence over some Sobolev balls, see for example Birg\'e $\&$ Massart \cite{BM97}, Barron, Birg\'e $\&$ Massart \cite{BBM99} or Massart \cite{Ma07}. In these cases, there exist $c>0$, $\alpha>0$ such that $\theta_n\geq cn^{\alpha}$. In more general situations, we can use the same trick as Arlot \cite{Ar08} and use our main theorem only for the models with dimension $d_m\geq (\ln n)^{4+2\gamma}$, they satisfy {\bf [BR]} with $\theta_n=(\ln n)^2$, at least when $n$ is sufficiently large, because
$$\|s\|^2+R_m\geq \|s\|^2+D_m\geq cd_m\geq c(\ln n)^4(\ln n)^{2\gamma}.$$ 
With our concentration inequalities, we can control easily the risk of the models with dimension $d_m\leq (\ln n)^{4+2\gamma}$ by $\kappa(\ln n)^{3+5\gamma/2}$ with probability larger than $1-Ce^{-\frac12(\ln n)^{\gamma}}$ and we can then deduce the following corollary.

\begin{coro}
Let $\M_n$ be either a collection of histograms satisfying Assumptions {\bf[PC]-[Reg]} or {\bf[PC]-[Ada]} or the collection of Fourier spaces of dimension $d_m\leq n$. There exist constants $\kappa>0$, $\eta>3+5\gamma/2$ and $C>0$ such that the estimator $\tilde{s}$ selected by a resampling penalty satisfies
$$\p\left(\|s-\tilde{s}\|^2\leq (1+\kappa(\ln n)^{-1})\left(\inf_{m\in\M_n}\|s-\hat{s}_m\|^2+\frac{(\ln n)^{\eta}}n\right)\right)\geq 1-Ce^{-\frac12(\ln n)^{\gamma}}.$$
\end{coro}

\section{Simulation study}\label{Ch2S4}
We propose in this section to show the practical performances of the slope algorithm and the resampling penalties on two examples. We estimate the density
$$s(x)=\frac34x^{-1/4}1_{[0,1]}(x)$$
and we compare the three following methods. 
\begin{enumerate}
\item The first one is the slope heuristic applied with the linear dimension $d_m$ of the models. We observe two main behaviors of $d_{\hat{m}(K)}$ with respect to $K$. Most of the times, we only observe one jump, as in Figure 1, and we find $K_{\min}$ easily.

\begin{figure}[!h]
\centering
\includegraphics[width=8cm,height=6cm]{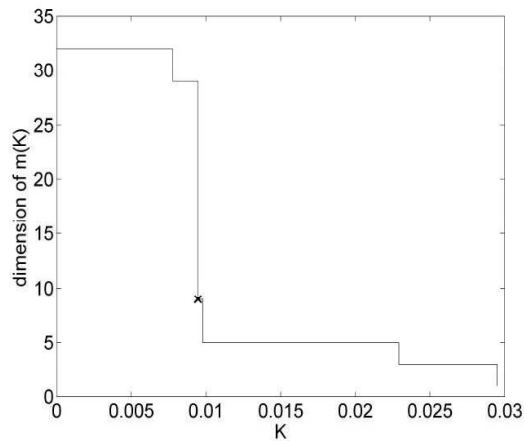}
\caption{Classical behavior of $K\mapsto d_{\hat{m}(K)}$}
\end{figure}

\noindent We also observe more difficult situations as the one of Figure 2 below, where we can see several jumps. In these cases, as prescribed in the regression framework by Arlot $\&$ Massart \cite{AM08}, we choose the constant $K_{\min}$ realizing the maximal jump of $d_{\hat{m}(K)}$. Arlot $\&$ Massart \cite{AM08} also proposed to select $K_{\min}$ as the minimal $K$ such that $d_{\hat{m}(K)}\leq d_{m^*}(\ln n)^{-1}$, but they obtained worse performances of the selected estimator in their simulations.\\
We justify this method only for collection of models where $d_m\simeq KD_m$ for some constant $K$. We will see that it gives really good performances when this condition is satisfied.
\item The second method is the resampling based penalization algorithm of Theorem \ref{OracleB1}. Note here that all the resampling penalties $D_m^W/n$ can be easily computed, without any Monte Carlo approximations. Actually, for all resampling scheme,
\begin{eqnarray*}
\frac{D_m^W}n&=&\frac{1}{n}\sum_{\lambda\in m}\left(P_n\psi_{\lambda}^2-\frac1{n(n-1)}\sum_{i\neq j=1}^n\psi_{\lambda}(X_i)\psi_{\lambda}(X_j)\right).
\end{eqnarray*}
Resampling penalties give always good approximations of $D_m$. However, in non asymptotic situations, it may be usefull to overpenalize a little bit in order to improve the leading constants in the oracle inequality (in Theorem \ref{P3}, imagine that $46\epsilon_n$ is very close to $1$).
\item In a third method, we propose therefore to use the slope algorithm applied with a complexity $D_m^W$. By this way, we hope to overpenalize a little bit the resampling penalty when it is necessary.
\end{enumerate}

\subsection{Example 1: regular case} 
In the first example, we consider the collection of regular histograms described in example {\bf HR} and we observe $n=100$ data. In this example, we saw that $D_m^W\simeq D_m\simeq d_m$. We can actually verify in Figure 2 that these quantities almost coincide for the selected model.

\begin{figure}[!h]
\centering
\includegraphics[width=12cm,height=6cm]{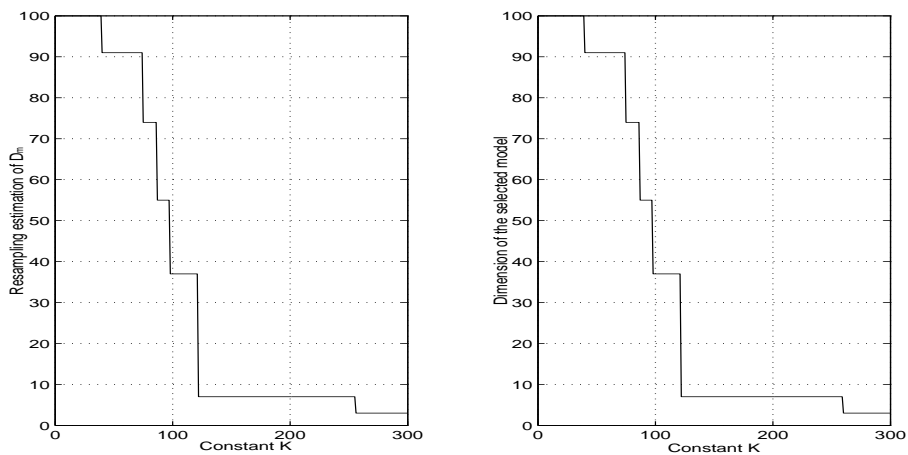}
\caption{Comparison of $d_m$ and $D_m^W$ on the selected model}
\end{figure}

\noindent We compute $N=1000$ times the oracle constant $c=\|s-\tilde{s}\|^2/(\inf_{m\in\M_n}\|s-\hat{s}_m\|^2)$ for the 3 methods. We put in the following array the mean, the median and the $0.95$-quantile, $q_{0.95}$ of these quantities.
\begin{center}
$\begin{array}{|c|c|c|c|}
\hline {\rm method} & {\rm mean}\; {\rm of}\;{\rm the}\; $N$\;{\rm constants}\;$c$ &{\rm median}&q_{0.95}\\
\hline {\rm slope+d_m} & 3.56 & 2.30 &10.07\\
\hline {\rm resampling} & 4.43 & 2.52 & 15.47\\
\hline {\rm resampling+slope} & 3.57 & 2.21 &10.86\\
\hline
\end{array}$
\end{center}
We observe that the slope algorithm allows to improve the resampling penalty in practice. This may be due to a little overpenalization even if it is not a straightforward consequence of our theoretical results. Note that, as $d_m\simeq D_m^W$, the slope algorithm leads to the same results when applied with $d_m$ or with $D_m^W$. Although we have an explicite formula to compute the resampling penalties, the computation time is much longer if we use $D_m^W$. Therefore, we clearly recommand to use the slope algorithm with $d_m$ for regular collections of model, as regular histograms or Fourier spaces described in Section \ref{Ch2S33}.
\subsection{Example 2: a more complicated collection} 
In the next example, we want to show that the linear dimension shall not be used in general. Let us consider a slightly more complicated collection. Let $k,J_1,J_2,n$ be four non null integers satisfying $k\leq n$, $J_1\leq k$, $J_2\leq n-k$. We denote by $S_{k,J_1,J_2,n}$ the linear space of histograms on the following partition.
\begin{eqnarray*}
&&\left\{\left[l\frac{k}{J_1n},(l+1)\frac{k}{J_1n}\right[,\; l=0,...,J_1-1\right\}\\
&&\cup\left\{\left[\frac kn+l\frac{1-k/n}{J_2},\frac kn+(l+1)\frac{1-k/n}{J_2}\right[,  \;l=0,...J_2-1\right\}.
\end{eqnarray*}
Let $n\in\N^*$ and let $\M_n=\{(k,J_1,J_2)\in(\N^*)^3;\;k\leq n,\;J_1\leq k,\;J_2\leq n-k\}$. It is clear that $\textrm{Card}(\M_n)\leq n^3$. The oracle of this collection is better than the previous one since the regular histograms belongs to $(S_{m,n})_{m\in\M_n}$. It is easy to check that the dimension of $S_{k,J_1,J_2,n}$ is equal to $J_1+J_2$ and that $D_{k,J_1,J_2,n}$ is equal to $(nJ_1/k)F(k/n)+(nJ_2/(n-k))(1-F(k/n))-\|s_{k,J_1,J_2,n}\|^2/n$, where $F$ is the distribution function of the observations. Hence, there is no constant $K_o$ such that $K_od_{k,J_1,J_2,n}\simeq D_{k,J_1,J_2,n}$ as in the previous example. Figure 3 let us see this fact on the selected model.

\begin{figure}[!h]
\centering
\includegraphics[width=12cm,height=8cm]{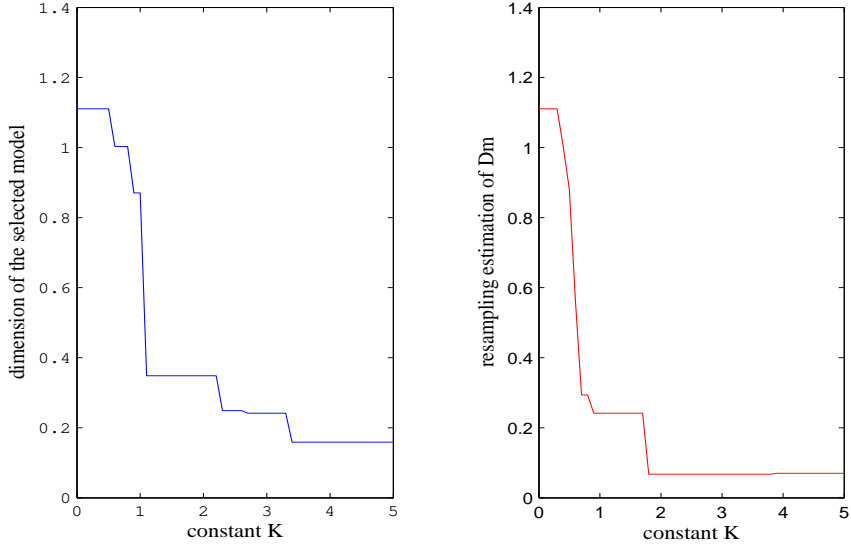}
\caption{Comparison of $d_m$ and $D_m^W$ on the selected model}
\end{figure}
\noindent
We also compute $N=1000$ times the oracle constant $c=\|s-\tilde{s}\|^2/(\inf_{m\in\M_n}\|s-\hat{s}_m\|^2)$ for the 3 methods, taking $n=100$ observations each time. The results are summarized in following array.
\begin{center}
$\begin{array}{|c|c|c|c|}
\hline {\rm method} & {\rm mean}\; {\rm of}\;{\rm the}\; $N$\;{\rm constants}\;$c$ &{\rm median}&q_{0.95}\\
\hline {\rm slope+d_m} & 8.30 & 7.01&19.73 \\
\hline {\rm resampling} & 6.11 & 5.08 &13.52\\
\hline {\rm resampling+slope} & 5.33 & 4.04 &12.92\\
\hline
\end{array}$
\end{center}
The slope heuristic gives bad results when applied with $d_m$. This is due to the fact that $d_m$ is not proportional to $D_m$ here. The resampling based penalty $2D_m^W/n$ is much better and, as in the regular case, it is well improved by the slope algorithm. Therefore, for general collections of models where we do not know an optimal shape of the ideal penalty, we recommand to apply the slope algorithm with a complexity equal to $D_m^W$.

\section{Proofs}\label{Ch2S5}
\subsection{Proof of Proposition \ref{concbasic}}
It is a straightforward application of Corollary \ref{C4} in the appendix.
\subsection{Technical lemmas}
Before giving the proofs of the main theorems, we state and prove some important technical lemmas that we will use repeatedly all along the proofs. Let us recall here the main notations. For all $m$, $m'$ in $\M_n$, 
$$p(m)=\|s_m-\hat{s}_m\|^2,\;D_m=n\E(p(m))=n\E\left(\|\hat{s}_m-s_m\|^2\right)$$
$$R_m=n\E\left(\|s-\hat{s}_m\|^2\right)=n\|s-s_m\|^2+D_m,\; \delta(m,m')=\nu_n(s_m-s_{m'}).$$
For all $n\in\N^*,\; k>0,\;k'>0,\;\gamma>0,$, let $[k]$ be the integer part of $k$ and let
$$l_{n,\gamma}(k,k')=\ln((1+{\rm Card(\M_n^{[k]})})(1+{\rm Card(\M_n^{[k']})}))+\ln((1+k)(1+k'))+(\ln n)^{\gamma}.$$
Recall that Assumption {\bf [V]} implies that, for all $m,m'$ in $\M_n$, 
\begin{eqnarray}\label{***}
v_{m,m'}^2l_{n,\gamma}(R_m,R_{m'})&\leq& \epsilon_n^2(R_m\vee R_{m'}),\nonumber\\
e_{m,m'}(l_{n,\gamma}(R_m,R_{m'}))^2&\leq& \epsilon^4_n(R_m\vee R_{m'}).
\end{eqnarray}
Let us  prove a simple result
\begin{lemma}\label{sommation}
For all $K>1$,
\begin{equation}\label{som1}
\Sigma(K)=\sum_{k\in \N}\sum_{m\in\M_n^k}e^{-K[\ln(1+{\rm Card(\M_n^k)})+\ln(1+k)]}<\infty.
\end{equation}
For all $m$ in $\M_n$, let $l_m=l_{n,\gamma}(R_{m},R_m)$, then, for all $K>1/\sqrt{2}$, 
\begin{equation}\label{som2}
\sum_{m\in\M_n}e^{-K^2l_{m}}=\Sigma(2K^2)e^{-K^2(\ln n)^{\gamma}}.
\end{equation}
For all $m$, $m'$ in $\M_n$, let $l_{m,m'}=l_{n,\gamma}(R_{m},R_{m'})$, then, for all $K>1$, 
\begin{equation}\label{som3}
\sum_{(m,m')\in(\M_n)^2}e^{-K^2l_{m,m'}}=(\Sigma(K^2))^2e^{-K^2(\ln n)^{\gamma}}.
\end{equation}
\end{lemma}
\begin{proof}:
Inequality (\ref{som1}) comes from the fact that, when $K>1$, 
$$\forall k\in\N,\;\sum_{m\in\M_n^k}e^{-K[\ln(1+{\rm Card}(\M_n^k))]}\leq 1,\;{\rm and}\;\sum_{k\in\N^*}e^{-K\ln k}<\infty.$$
For all integer $k$ such that $\M_n^k\neq \emptyset$, for all $m$ in $\M_n^k$, $l_m\geq 2[\ln(1+{\rm Card}(\M_n^k))+\ln(1+k)]+(\ln n)^{\gamma}$, thus, for all $K>1/\sqrt{2}$, it comes from (\ref{som1}) that
$$\sum_{m\in\M_n}e^{-K^2l_m}\leq e^{-K^2(\ln n)^{\gamma}}\sum_{k\in\N}\sum_{m\in\M_n^k}e^{-2K^2[\ln(1+{\rm Card}(\M_n^k))+\ln(1+k)]}\leq \Sigma(2K^2)e^{-K^2(\ln n)^{\gamma}}.$$
Finally, for all integers $(k,k')$ such that $\M_n^k\times \M_n^{k'}\neq \emptyset$, 
$$l_{m,m'}\geq \ln(1+{\rm Card}(\M_n^k))+\ln(1+k)+\ln(1+{\rm Card}(\M_n^{k'}))+\ln(1+k')+(\ln n)^{\gamma}.$$
Thus, from (\ref{som1}),
$$\sum_{(m,m')\in(\M_n^2)}e^{-K^2l_{m,m'}}=\left(\sum_{k\in\N}\sum_{m\in\M_n^k}e^{-K^2[\ln(1+{\rm Card(\M_n^k)})+\ln(1+k)]}\right)^2e^{-K^2(\ln n)^{\gamma}}.$$
\end{proof}

\begin{lemma}\label{events}
Let $\M_n$ be a collection of models satisfying Assumption {\bf [V]}. We consider the following events.
\begin{eqnarray*}
\Omega_{\delta}&=&\left\{\forall (m,m')\in\M_n^2,\; \delta(m,m')\leq 6\epsilon_n\frac{R_m\vee R_{m'}}{n}\right\}\\
\Omega_{p}&=&\bigcap_{m\in\M_n}\left\{\left\{p(m)-\frac{D_m}n\leq 10\epsilon_n\frac{R_m}{n}\right\}\cap\left\{p(m)-\frac{D_m}n\geq- 20\epsilon_n\frac{R_m}{n}\right\}\right\}
\end{eqnarray*}
and $\Omega_T=\Omega_{\delta}\cap\Omega_{p}$. Then there exists a constant $C>0$ such that 
$$\p(\Omega_{\delta}^c)\leq Ce^{-(\ln n)^{\gamma}},\;\p(\Omega_{p}^c)\leq Ce^{-\frac12(\ln n)^{\gamma}},\;\p(\Omega_T^c)\leq Ce^{-\frac12(\ln n)^{\gamma}}.$$
\end{lemma}
\begin{proof}: Let $K>1$ be a constant to be chosen later. We apply Lemma \ref{TL2} in the appendix to $u=s_m-s_{m'}$, $S=S_m+S_{m'}$, $L=id$, $x=K^2l_{n,\gamma}(R_m,R_{m'})$. For all $\eta>0$, for all $m,m'$ in $\M_n$, on an event of probability larger than $1-e^{-K^2l_{n,\gamma}(R_m,R_{m'})}$, 
\begin{equation}\label{d}
\delta(m,m')\leq \frac{\eta}2\|s_m-s_{m'}\|^2+\frac{2v_{m,m'}^2K^2l_{n,\gamma}(R_m,R_{m'})+e_{m,m'}(K^2l_{n,\gamma}(R_m,R_{m'}))^2/9}{\eta n}.
\end{equation}
From {\bf[V]}, for all $m$, $m'$ in $\M_n$,
\begin{equation*}
2v_{m,m'}^2K^2l_{n,\gamma}(R_m,R_{m'}))+\frac{e_{m,m'}(K^2l_{n,\gamma}(R_m,R_{m'}))^2}{9}\leq\left(2(K\epsilon_n)^2+\frac{(K\epsilon_n)^4}{9}\right)\frac{R_m\vee R_{m'}}n.
\end{equation*}
Moreover, for all $m,m'$ in $\M_n$, 
$$\|s_m-s_{m'}\|^2\leq 2(\|s-s_m\|^2+\|s-s_{m'}\|^2)\leq 2(R_m+R_{m'})\leq 4(R_m\vee R_{m'}).$$ 
Let $e_n(K)=\sqrt{(K\epsilon_n)^2+(K\epsilon_n)^4/18}$. In (\ref{d}) we take $\eta=e_n(K)$ and we obtain
\begin{equation}\label{d2}
\p\left(\delta(m,m')> 4e_n(K)\frac{R_m\vee R_{m'}}n\right)\leq e^{-Kl_{n,\gamma}(R_m,R_{m'})}.
\end{equation}
From (\ref{som3}), for all $K>1$,
$$\p\left(\forall (m,m')\in\M_n^2,\;\delta(m,m')> 4e_n(K)\frac{R_m\vee R_{m'}}n\right)\leq (\Sigma(K))^2e^{-K(\ln n)^2}.$$
Let $K=1.1$ and take $n$ sufficiently large so that $K^4\epsilon_n^2/18\leq 1$, then $4 e_n(K)\leq 6\epsilon_n$. Hence, the first conclusion of Lemma \ref{events} holds for sufficiently large $n$, it holds in general, provided that we increase the constant $C$ if necessary.\\
We apply Assumption {\bf[V]} (see (\ref{***})) with $m=m'$, let $l_m=l_{n,\gamma}(R_m,R_m)$, for all $K>0$, for all $n$ such that $4.06(K\epsilon_n)^3\leq 2$,
\begin{eqnarray*}
&&\frac{D_m^{3/4}(e_m(K^2l_m)^2)^{1/4}+0.7\sqrt{D_mv_m^2K^2l_m}+0.15v_m^2K^2l_m+e_m (K^2l_m)^2}n\\
&&\leq (1.7K\epsilon_n+0.15(K\epsilon_n)^2+(K\epsilon_n)^4)\frac{R_m}n\leq3K\epsilon_n\frac{R_m}n.
\end{eqnarray*}
\begin{eqnarray*}
&&\frac{1.8D_m^{3/4}(e_m(K^2l_m)^2)^{1/4}+1.71\sqrt{D_mv_m^2(K^2l_m)}+4.06e_m (K^2l_m)^2}n\\
&&\leq (3.51K\epsilon_n+4.06(K\epsilon_n)^4)\frac{R_m}n\leq6K\epsilon_n\frac{R_m}n.
\end{eqnarray*}
It comes then from Proposition \ref{concbasic} applied with $x=K^2l_m$ that, for all $m$ in $\M_n$
$$\p\left(p(m)-\frac{D_m}n> 3K\epsilon_n\frac{R_m}n\right)\leq e^{-\frac{K^2}{20}l_m}.$$
Thus, from (\ref{som2}), for all $K>\sqrt{10}$, and for all $n$ sufficiently large,
$$\p\left(\forall m\in\M_n,\; p(m)-\frac{D_m}n> 3K\epsilon_n\frac{R_m}n\right)\leq \Sigma(K^2/10)e^{-\frac{K^2}{20}(\ln n)^{\gamma}}.$$
We use the same arguments to prove that
$$\p\left(\forall m\in\M_n,\; p(m)-\frac{D_m}n< 6K\epsilon_n\frac{R_m}n\right)\leq \Sigma(K^2/10)e^{-\frac{K^2}{20}(\ln n)^{\gamma}}.$$
Fixe $K=\sqrt{10.5}$, then for all $n$ sufficiently large , the conclusion of Lemma \ref{events} holds. It holds in general provided that we increase the constant $C$ if necessary.
\end{proof}

\begin{lemma}\label{TL3}
Let $(\psi_{\lambda})_{\lambda\in\Lambda}$ be an orthonormal system in $L^2(\mu)$ and let $L$ be a linear functional defined on $L^2(\mu)$. Let $p(\Lambda)=\sum_{\lambda\in\Lambda}(\nu_n(L(\psi_{\lambda})))^2$. Let $(W_1,...,W_n)$ be a resampling scheme, let $\bar{W}_n=\sum_{i=1}^nW_i/n$ and let $v_W^2=\Var(W_1-\bar{W}_n)$. Let $$D_{\Lambda}^W=n(v_W^2)^{-1}\sum_{\lambda\in\Lambda}\E^W\left((\nu_n^W(L(\psi_{\lambda})))^2\right),$$ $T=\sum_{\lambda\in\Lambda}(L(\psi_{\lambda})-PL(\psi_{\lambda}))^2$, $D=PT$ and
$$U=\frac1{n(n-1)}\sum_{i\neq j=1}^n\sum_{\lambda\in\Lambda}(L(\psi_{\lambda})(X_i)-PL(\psi_{\lambda}))(L(\psi_{\lambda})(X_j)-PL(\psi_{\lambda})).$$
then 
$$p(\Lambda)=\frac 1nP_nT+\frac{n-1}nU,\;D_{\Lambda}^W=P_nT-U,\;p(\Lambda)-\frac{D_{\Lambda}^W}n=U,$$
$$\E(D_{\Lambda}^W)=D,\;D_{\Lambda}^W-D=\nu_nT-U.$$
\end{lemma}

\begin{proof}: 
It is easy to check that
\begin{eqnarray*}
p(\Lambda)&=&\sum_{\lambda\in\Lambda}(\frac 1n\sum_{i=1}^nL(\psi_{\lambda})(X_i)-PL(\psi_{\lambda}))^2=\frac1{n^2}\sum_{i=1}^n(L(\psi_{\lambda})(X_i)-PL(\psi_{\lambda}))^2\\
&&+\frac{1}{n^2}\sum_{i\neq j=1}^n\sum_{\lambda\in\Lambda}(L(\psi_{\lambda})(X_i)-PL(\psi_{\lambda}))(L(\psi_{\lambda})(X_j)-PL(\psi_{\lambda}))\\
&=&\frac 1nP_nT+\frac{n-1}nU.
\end{eqnarray*}
Recall that $\nu_n^W=P_n^W-\bar{W}_nP_n$. For all $\lambda$ in $\Lambda$, since $\sum_{i=1}^n(W_i-\bar{W}_n)=0$,
\begin{eqnarray*}
\nu_n^W(L(\psi_{\lambda}))&=&\frac1n\sum_{i=1}^n(W_i-\bar{W}_n)L(\psi_{\lambda})(X_i)\\
&=&\frac1n\sum_{i=1}^n(W_i-\bar{W}_n)(L(\psi_{\lambda})(X_i)-PL(\psi_{\lambda})). 
\end{eqnarray*}
Thus, if $E_{i,j}=\E\left((W_i-\bar{W}_n)(W_j-\bar{W}_n)\right)/v_W^2$,
\begin{eqnarray*}
D_{\Lambda}^W&=&n(v_W^2)^{-1}\sum_{\lambda\in\Lambda}\E^W\left(\left(\frac 1n\sum_{i=1}^n(W_i-\bar{W}_n)(L(\psi_{\lambda})(X_i)-PL(\psi_{\lambda}))\right)^2\right)\\
&=&\frac{1}{n}\sum_{i=1}^n\frac{\E\left((W_i-\bar{W}_n)^2\right)}{v_W^2}(L(\psi_{\lambda})(X_i)-PL(\psi_{\lambda}))^2+\\
&&\frac{1}{n}\sum_{i\neq j=1}^n\sum_{\lambda\in\Lambda}E_{i,j}(L(\psi_{\lambda})(X_i)-PL(\psi_{\lambda}))(L(\psi_{\lambda})(X_j)-PL(\psi_{\lambda})).
\end{eqnarray*}
Since the weights are exchangeable, for all $i=1,..,n$, $\E((W_i-\bar{W}_n)^2)=\Var(W_1-\bar{W}_n)=v_W^2$ and for all $i\neq j=1,...,n$, 
$$v_W^2E_{i,j}=\E\left((W_i-\bar{W}_n)(W_j-\bar{W}_n)\right)=\E\left((W_1-\bar{W}_n)(W_2-\bar{W}_n)\right).$$ Moreover, since $\sum_{i=1}^n(W_i-\bar{W}_n)=0$, 
\begin{eqnarray*}
0&=&E\left[\left(\sum_{i=1}^n(W_i-\bar{W}_n)\right)^2\right]=\sum_{i=1}^n\E\left((W_i-\bar{W}_n)^2\right)+\sum_{i\neq j=1}^nv_W^2E_{i,j}\\
&=&n\E((W_1-\bar{W}_n)^2)+n(n-1)\E\left((W_1-\bar{W}_n)(W_2-\bar{W}_n)\right).
\end{eqnarray*}
Hence, for all $i\neq j=1,...,n$, $E_{i,j}=-1/(n-1)$, thus
$$D_{\Lambda}^W=P_nT-U.$$
The last inequalities of Lemma \ref{TL3} follow from the fact that $\E(U)=0$. Finally,
$$p(\Lambda)-\frac{D_{\Lambda}^W}n=\frac 1nP_nT+\frac{n-1}nU-\left(\frac 1nP_nT-\frac{1}nU\right)=U.$$
\end{proof}

\begin{lemma}\label{eventsB}
 Let
\begin{eqnarray*}
\Omega_{u}&=&\bigcap_{m\in\M_n}\left\{\frac{D_m^W}n-p(m)\leq10\epsilon_n\frac{R_m}{n}\right\}\\
\Omega_{l}&=&\bigcap_{m\in\M_n}\left\{\frac{D_m^W}n-p(m)\geq-12\epsilon_n\frac{R_m}{n}\right\} 
\end{eqnarray*}
and $\tilde{\Omega}_{p}=\Omega_u\cap\Omega_l$. There exists a constant $C>0$ such that $\p(\tilde{\Omega}_{p}^c)\leq C e^{-\frac12(\ln n)^{\gamma}}$.
\end{lemma}
\begin{proof}:
From Assumption {\bf [V]} applied with $m=m'$, (see (\ref{***})), if $l_m=l_{n,\gamma}(R_m,R_m)$, for all $K>0$,
$$D_m^{3/4}(e_m (K^2l_m)^2)^{1/4}\leq K\epsilon_n R_m,\;\sqrt{v_m^2D_m(K^2l_m)}\leq K\epsilon_n R_m,$$
$$v_m^2(K^2l_m)\leq (K\epsilon_n)^2R_m,\;e_m (Kl_m)^2\leq (K\epsilon_n)^4R_m.$$
We apply Proposition \ref{cboot} with $x=K^2l_m$ and we obtain
\begin{equation*}
\p\left(\frac{D_m^W}n-p(m)>\left(8.31K\epsilon_n+3(K\epsilon_n)^2+(19.1)^2(K\epsilon_n)^4\right)\frac{R_m}{n-1}\right)\leq 2e^{-K^2l_m}.
\end{equation*}
Thus, for all $K>1/(\sqrt{2})$, if $e_n(K)=n\left(8.31K\epsilon_n+3(K\epsilon_n)^2+(19.1)^2(K\epsilon_n)^4\right)/(n-1)$, from (\ref{som2})
\begin{equation*}
\p\left(\forall m\in \M_n,\;\frac{D_m^W}n-p(m)>e_n(K)\frac{R_m}{n}\right)\leq 2\Sigma(2K^2)e^{-K^2(\ln n)^{\gamma}}.
\end{equation*}
Take $K=8/8.31$ and $n\geq 10$ sufficiently large to ensure that $3K^2\epsilon_n+(19.1)^2K^4\epsilon_n^3\leq 1$, then 
$$e_n(K)\leq \frac{10}9\left(8\epsilon_n+\epsilon_n\right)\leq 10\epsilon_n.$$
We deduce that, for sufficiently large $n$,
$$\p(\Omega_u^c)\leq 2\Sigma(2K^2)e^{-K^2(\ln n)^{\gamma}}.$$
We also apply Proposition \ref{cboot} with $x=K^2l_m$, and we use the same arguments to prove that, for $K=16/16.61$, for all $n\geq 10$ sufficiently large to ensure that $(40.3)^2K^4\epsilon_n^3\leq 2$
\begin{equation*}
\p\left(\forall m\in \M_n,\;\frac{D_m^W}n-p(m)<-20\epsilon_n\frac{R_m}{n}\right)\leq 3.8\Sigma(2K^2)e^{-K^2(\ln n)^{\gamma}}.
\end{equation*}
Hence, the conclusion of Lemma \ref{eventsB} holds for sufficiently large $n$. It holds in general, provided that we increase the constant $C$ if necessary.
\end{proof}

\subsection{Proof of Theorem \ref{P2}}
If $c_n<0$, there is nothing to prove. We can then assume that $c_n\geq 0$, this implies in particular that 
$$28\epsilon_n\leq \delta_n<1.$$
We use the notations of Lemma \ref{events}. From Lemma \ref{events}, the inequalities (\ref{slope1}) will be proved if, on $\Omega_T$, $D_{\hat{m}}\geq c_nD_{m^*}$ and
$$\|s-\tilde{s}\|^2\geq \frac{c_n}{5h_n^o}\inf_{m\in\M_n}\|s-\hat{s}_m\|^2.$$
Let $m_o\in\arg\min_{m\in\M_n}R_m$, $\hat{m}$ minimizes over $\M_n$ the following criterion.
\begin{eqnarray*}
\textrm{Crit}(m)&=&P_nQ(\hat{s}_m)+\pen(m)+\|s\|^2+2\nu_n(s_{m_o})\\
&=&\|s-s_m\|^2-p(m)+\delta(m_o,m)+\pen(m).
\end{eqnarray*}
Recall that $0\leq \pen(m)\leq (1-\delta_n)D_m/n$.
On $\Omega_T$, for all $m$ in $\M_n$, since $R_m\geq R_{m_o}$, 
\begin{eqnarray*}
\textrm{Crit}(m)&\geq&\|s-s_m\|^2-\frac{D_m}n-16\epsilon_n\frac{R_m}n\geq -(1+16\epsilon_n)\frac{D_m}n.\\
\textrm{Crit}(m)&\leq& \|s-s_m\|^2+26\epsilon_n\frac{R_m}n-\delta_n\frac{D_m}n=(1+26\epsilon_n)\|s-s_m\|^2-(\delta_n-26\epsilon_n)\frac{D_m}n.\\
\end{eqnarray*}
When $D_m\leq c_nD_{m^*}$,
$$(1+16\epsilon_n)D_{m}\leq D_{m^*}\left( (\delta_n-26\epsilon_n)-(1+26\epsilon_n)\frac{n\|s-s_{m^*}\|^2}{D_{m^*}}\right).$$
Thus ${\rm Crit}(m)\geq {\rm Crit}(m^*)$. This implies that $D_{\hat{m}}\geq c_nD_{m^*}$. \\
Moreover, on $\Omega_T$, we also have, for all $m$ in $\M_n$
$$\|s-\tilde{s}\|^2=\frac{R_{\hat{m}}}n+\left(p(\hat{m})-\frac{D_{\hat{m}}}n\right)\geq (1-20\epsilon_n)\frac{R_{\hat{m}}}n,$$
and
$$\inf_{m\in\M_n}\|s-\hat{s}_m\|^2\leq \inf_{m\in\M_n}\frac{R_m}n(1+10\epsilon_n)\leq \frac{R_{m_o}}n(1+10\epsilon_n).$$
Thus 
\begin{eqnarray*}
\|s-\tilde{s}\|^2&\geq& (1-20\epsilon_n)\frac{R_{\hat{m}}}n\geq (1-20\epsilon_n)\frac{D_{\hat{m}}}n\geq (1-20\epsilon_n)c_n\frac{D_{m^*}}n\\
&\geq&c_n\frac{1-20\epsilon_n}{h_n^o}\frac{R_{m_o}}{n}\geq \frac{c_n}{h_n^o}\frac{1-20\epsilon_n}{1+10\epsilon_n}\inf_{m\in\M_n}\|s-\hat{s}_m\|^2.
\end{eqnarray*}
We conclude the proof, saying that $\epsilon_n\leq 1/28$ implies that $(1-20\epsilon_n)(1+10\epsilon_n)^{-1}\geq 8/38\geq 1/5.$
\subsection{Proof of Theorem \ref{P3}}
If $\delta_--46\epsilon_n<-1$, there is nothing to prove, hence, we can assume in the following that $\delta_--46\epsilon_n>-1$.\\
We keep the notation $\Omega_T$ introduced in Lemma \ref{events}. Let 
$$\Omega_{\pen}=\bigcap_{m\in\M_n}\left\{\frac{2D_m}n+\delta_-\frac{R_m}n\leq \pen(m)\leq \frac{2D_m}n+\delta^+\frac{R_m}n\right\},$$
$\Omega=\Omega_T\cap\Omega_{\pen}$ and $m_o\in\arg\min_{m\in\M_n}R_m$. Recall that $\p(\Omega_{\pen})\geq 1-p'$ and that, $\hat{m}$ minimizes over $m$ the following criterion.
\begin{eqnarray*}
\textrm{Crit}(m)&=&P_nQ(\hat{s}_m)+\pen(m)+\|s\|^2+2\nu_n(s_{m_o})\\
&=&\|s-s_m\|^2-p(m)+\delta(m_o,m)+\pen(m).
\end{eqnarray*}
Therefore, on $\Omega$, for all $m$ in $\M_n$, since $R_m\geq R_{m_o}$,
\begin{eqnarray*}
\textrm{Crit}(m)&\geq&(1+\delta_-)\frac{R_m}n+\left(\frac{D_m}n-p(m)\right)-6\epsilon_n\frac{R_{m}}n\\
&\geq&(1+\delta_--16\epsilon_n)\|s-s_m\|^2+(1+\delta_--16\epsilon_n)\frac{D_m}n\geq (1+\delta_--16\epsilon_n)\frac{D_m}n\\
\textrm{Crit}(m)&\leq& (1+\delta^++26\epsilon_n)\frac{R_m}n.
\end{eqnarray*}
If $D_m> C_n(\delta_-,\delta^+)R_{m_o}$,
$$(1+\delta_--16\epsilon_n)D_m> (1+\delta^++26\epsilon_n)R_{m_o},$$
Thus ${\rm Crit}(m)> {\rm Crit}(m_o)$, hence $D_{\hat{m}}\leq C_n(\delta_-,\delta^+)R_{m_o}$.\\
Moreover, from (\ref{decRis}), for all $m$ in $\M_n$
\begin{eqnarray*}
\|s-\tilde{s}\|^2&\leq&\|s-\hat{s}_m\|^2+(\pen(m)-2p(m))+(2p(\hat{m})-\pen(\hat{m}))+\delta(\hat{m},m)\\
&\leq&\|s-\hat{s}_m\|^2+2\left(\frac{D_m}n-p(m)\right)+(\delta^++6\epsilon_n)\frac{R_m}n\\
&&+2\left(p(\hat{m})-\frac{D_{\hat{m}}}n\right)+(-\delta_-+6\epsilon_n)\frac{R_{\hat{m}}}n\\
&\leq&\|s-\hat{s}_m\|^2+(46\epsilon_n+\delta^+)\frac{R_{m}}n+(26\epsilon_n-\delta_-)\frac{R_{\hat{m}}}n.
\end{eqnarray*}
For all $m$ in $\M_n$, on $\Omega_T$,
$$\|s-\hat{s}_m\|^2=\frac{R_m}n+\left(p(m)-\frac{D_m}n\right)\geq (1-20\epsilon_n)\frac{R_m}n.$$
Hence, for all $m\in\M_n$,
$$\|s-\tilde{s}\|^2\leq \|s-\hat{s}_m\|^2\left(1+\frac{46\epsilon_n+\delta^+}{1-20\epsilon_n}\right)+\frac{26\epsilon_n-\delta_-}{1-20\epsilon_n}\|s-\tilde{s}\|^2.$$ 
This concludes the proof of Proposition \ref{P3}.

\subsection{Proof of Proposition \ref{cboot}}
We apply Lemma \ref{TL3} with $L=id$ and $\Lambda=m$. By definition of $p(m)$ and $D_m^W$, 
$$p(m)-\frac{D_m^W}n=\frac1{n(n-1)}\sum_{i\neq j=1}^n\sum_{\lambda\in m}(\psi_{\lambda}(X_i)-P\psi_{\lambda})(\psi_{\lambda}(X_j)-P\psi_{\lambda}).$$
Thus, from Lemma \ref{TL1} in the appendix, for all $x>0$, 
\begin{equation*}
\p\left( p(m)-\frac{D_m^W}n>\frac{5.31D_m^{3/4}(e_m x^2)^{1/4}+3\sqrt{v_m^2D_mx}+3v_m^2x+e_m (19.1x)^2}{n-1}\right)\leq2e^{-x}.
\end{equation*}
\begin{equation*}
\p\left(\frac{D_m^W}n- p(m)>\frac{9D_m^{3/4}(e_m x^2)^{1/4}+7.61\sqrt{v_m^2D_mx}+e_m (40.3x)^2}{n-1}\right)\leq3.8e^{-x}.
\end{equation*}
This proves (\ref{concboot3}) and (\ref{concboot4}). \\
In order to obtain (\ref{concboot1}) and (\ref{concboot2}), we introduce, for all $m$ in $\M_n$, the function $T_m=\sum_{\lambda\in m}(\psi_{\lambda}-P\psi_{\lambda})^2$ and the random variable 
$$U_m=\frac1{n(n-1)}\sum_{i\neq j=1}^n\sum_{\lambda\in m}(\psi_{\lambda}(X_i)-P\psi_{\lambda})(\psi_{\lambda}(X_j)-P\psi_{\lambda}).$$
We apply Lemma \ref{TL3} with $L=id$, we obtain
$$D_m^W-D_m=\nu_n(T_m)-U_m.$$
From Bernstein's inequality (see Proposition \ref{BI}), for all $x>0$ and all $\xi$ in $\{-1,1\}$,
\begin{equation*}
\p\left(\xi\nu_n(T_m)>\sqrt{\frac{2\Var(T_m(X))x}n}+\frac{\left\|T_m\right\|_{\infty}x}{3n}\right)\leq e^{-x}.
\end{equation*}
From Cauchy-Schwarz inequality, $T_m=\sup_{t\in B_m}(t-Pt)^2$, thus $\left\|T_m\right\|_{\infty}/n=4e_m$ and $\Var(T_m(X))/n\leq \left\|T_m\right\|_{\infty}PT_m/n=4e_mD_m$, therefore, for all $x>0$ and all $\xi$ in $\{-1,1\}$,
\begin{equation*}
\p\left(\xi\nu_n(T_m)>\sqrt{8e_mD_mx}+\frac{4e_m x}{3}\right)\leq e^{-x}.
\end{equation*}
Moreover, from Lemma \ref{TL1} in the appendix, for all $x>0$,
\begin{equation*}
\p\left( U_m>\frac{5.31D_m^{3/4}(e_m x^2)^{1/4}+3\sqrt{v_m^2D_mx}+3v_m^2x+e_m (19.1x)^2}{n-1}\right)\leq2e^{-x}.
\end{equation*}
\begin{equation*}
\p\left( U_m<-\frac{9D_m^{3/4}(e_m x^2)^{1/4}+7.61\sqrt{v_m^2D_mx}+e_m (40.3x)^2}{n-1}\right)\leq3.8e^{-x}.
\end{equation*}
We deduce that, for all $x>0$, with probability larger than $1-4.8e^{-x}$,
\begin{eqnarray*}
D_m^W-D_m&\leq &\sqrt{8e_mD_mx}+e_m\left(\frac{4x}3+\frac{ (40.3x)^2}{n-1}\right)\\
&&+\frac{9D_m^{3/4}(e_m x^2)^{1/4}+7.61\sqrt{v_m^2D_mx}}{n-1}.
\end{eqnarray*}
Moreover, for all $x>0$, on an event of probability larger than $1-3e^{-x}$,
\begin{eqnarray*}
D_m^W-D_m&\geq&-\sqrt{8e_mD_mx}-e_m\left(\frac{4x}3+\frac{ (19.1x)^2}{n-1}\right)\\
&&-\frac{5.31D_m^{3/4}(e_m x^2)^{1/4}+3\sqrt{v_m^2D_mx}+3v_m^2x}{n-1}.
\end{eqnarray*}

\subsection{Proof of Theorem \ref{OracleB1}}
Recall that $\p\left(\Omega_T^c\right)\leq C e^{-\frac12(\ln n)^{\gamma}},$ and that, on $\Omega_T$, 
$$\forall m\in\M_n, (1-20\epsilon_n)\frac{R_m}n\leq \|s-\hat{s}_m\|^2,$$ 
$$\forall m,m'\in\M_n^2,\; \delta(m,m')\leq 6\epsilon_n\frac{R_m\vee R_{m'}}{n}.$$
Let $\tilde{\Omega}_p$ be the event defined in Lemma \ref{eventsB} and let $\Omega=\tilde{\Omega}_p\cap \Omega_{T}$, from Lemma \ref{events}, $\p\left(\Omega^c\right)\leq C e^{-\frac12(\ln n)^{\gamma}}$. Recall that $\pen(m)=2D_m^W/n$. On $\Omega$, from (\ref{decRis}), for all $n$ such that $20\epsilon_n<1$, for all $m$ in $\M_n$, 
\begin{eqnarray*}
\|s-\tilde{s}\|^2&\leq&\|s-\hat{s}_m\|^2+26\epsilon_n\frac{R_m}{n}+16\epsilon_n\frac{R_{\hat{m}}}{n}\\
&\leq&\|s-\hat{s}_m\|^2+\frac{26\epsilon_n}{1-20\epsilon_n}\|s-\hat{s}_m\|^2+\frac{16\epsilon_n}{1-20\epsilon_n}\|s-\tilde{s}\|^2.
\end{eqnarray*}
Hence, for all $n$ such that $20\epsilon_n<1$, on $\Omega$,
$$(1-36\epsilon_n)\|s-\tilde{s}\|^2\leq (1+6\epsilon_n)\inf_{m\in\M_n}\|s-\hat{s}_m\|^2.$$
For all $n$ such that $42/(1-36\epsilon_n)<100$,
$$\|s-\tilde{s}\|^2\leq \left(1+\frac{42\epsilon_n}{1-36\epsilon_n}\right)\inf_{m\in\M_n}\|s-\hat{s}_m\|^2\leq(1+100\epsilon_n)\inf_{m\in\M_n}\|s-\hat{s}_m\|^2.$$
Hence (\ref{Or1}) holds for sufficiently large $n$, it holds in general provided that we enlarge the constant $C$ if necessary..

\section{Appendix}\label{Ch2S6}
In this Section, we state and prove some technical lemmas that are useful in the proofs. The main tool is the first Lemma based on Bousquet's version of Talagrand's inequality. It is a concentration inequality for the square of the supremum of the empirical process over a uniformly bounded class of functions. Recall first Bousquet's \cite{Bo02} and Klein $\&$ Rio \cite{KR05} versions of Talagrand's inequality.
\begin{theo}(Bousquet's bound)
Let $X_1,...,X_n$ be i.i.d. random variables valued in a measurable space $(\x,\mathcal{X})$ and let $S$ be a class of real valued functions bounded by $b$. Let $v^2=\sup_{t\in S}\Var(t(X))$ and let $Z=\sup_{t\in S}\nu_nt$. Then 
$$\forall x>0,\;\p\left(Z>\E(Z)+\sqrt{\frac{2}n(v^2+2b\E(Z))x}+\frac{bx}{3n}\right)\leq e^{-x}.$$
\end{theo}
\begin{theo}(Klein $\&$ Rio's bound)
Let $X_1,...,X_n$ be i.i.d. random variables valued in a measurable space $(\x,\mathcal{X})$ and let $S$ be a class of real valued functions bounded by $b$. Let $v^2=\sup_{t\in S}\Var(t(X))$ and let $Z=\sup_{t\in S}\nu_nt$. Then 
$$\forall x>0,\;\p\left(Z<\E(Z)-\sqrt{\frac{2}n(v^2+2b\E(Z))x}-\frac{8bx}{3n}\right)\leq e^{-x}.$$
\end{theo}
Let us now also recall Bernstein's inequality.
\begin{prop}\label{BI}{Bernstein's inequality}\\
Let $X_1,...,X_n$ be iid random variables valued in a measurable space $(X,\mathcal{X})$ and let $t$ be a measurable real valued function. Then, for all $x>0$,
$$\p\left(\nu_n(t)>\sqrt{\frac{2\Var (t(X_1))x}{n}}+\frac{\left\|t\right\|_{\infty}x}{3n}\right)\leq e^{-x}.$$
\end{prop}
We derive from these bounds the following useful corollary. Hereafter, $S$ denotes a symetric class of real valued functions upper bounded by $b$, $v^2=\sup_{t\in S}\Var(t(X))$, $Z=\sup_{t\in S}\nu_nt$, $n\E(Z^2)=D$. Since $S$ is symetric, we always have $Z\geq 0$.
\begin{coro}\label{C1}
Let $S$ be a symetric class of real valued functions upper bounded by $b$, $v^2=\sup_{t\in S}\Var(t(X))$, $Z=\sup_{t\in S}\nu_nt$, $n\E(Z^2)=D$, $e_b=b^2/n$ and 
$$nE_m=225e_b+\left(2.1+\sqrt{2\pi}\right)\sqrt{v^2D}+\sqrt{15}D^{3/4}e_b^{1/4},$$
then
\begin{equation}\label{E1}
\E(Z^2{\bf 1}_{Z\geq\E(Z)})\leq (\E(Z))^2\p\left( Z\geq\E(Z)\right)+E_m.
\end{equation}
In particular,
\begin{equation}
(\E(Z))^2\leq \E(Z^2)\leq (\E(Z))^2+E_m.
\end{equation}
\end{coro}
\begin{proof}:
We have 
\begin{eqnarray*}
\E(Z^2{\bf 1}_{Z\geq\E(Z)})&=&\int_0^{\infty}\p(Z^2{\bf 1}_{Z\geq\E(Z)}>x)dx=\int_0^{\infty}\p(Z{\bf 1}_{Z\geq\E(Z)}>\sqrt{x})dx\\
&=&(\E(Z))^2\p\left( Z\geq\E(Z)\right)+\int_{(\E(Z))^2}^{\infty}\p(Z>\sqrt{x})dx
\end{eqnarray*}
Take $x=(\E(Z)+\sqrt{2(v^2+2b\E(Z))y/n}+by/(3n))^2$ in the previous integral, from Bousquet's version of Talagrand's inequality,
\begin{eqnarray*}
\E(Z^2{\bf 1}_{Z\geq\E(Z)})&\leq&\E(Z)\sqrt{\frac2{n}(v^2+2b\E(Z))}\int_0^{\infty}\frac{e^{-y}}{\sqrt{y}}dy+\frac{2v^2+14b\E(Z)/3}n\int_0^{\infty}e^{-y}dy\\
&&+\frac{b}{n}\sqrt{\frac{2}n(v^2+2b\E(Z))}\int_0^{\infty}e^{-y}\sqrt{y}dy+\frac{2b^2}{9n^2}\int_{0}^{\infty}ye^{-y}dy.
\end{eqnarray*}
Classical computations lead to 
$$\int_0^{\infty}\frac{e^{-y}}{\sqrt{y}}dy=2\int_0^{\infty}e^{-y}\sqrt{y}dy=\sqrt{\pi},\;\int_0^{\infty}e^{-y}dy=\int_{0}^{\infty}ye^{-y}dy=1.$$
Therefore, if $e_b=b^2/n$, using repeatedly the inequalities
\begin{equation}\label{**}
a^{\alpha}b^{1-\alpha}\leq\alpha a+(1-\alpha)b 
\end{equation}
and $\sqrt{a+b}\leq \sqrt{a}+\sqrt{b}$, we obtain, for all $\eta>0$,
$$\sqrt{ne_b}\E(Z)\leq \frac{e_b}{3\eta^2}+\frac{2\eta}{3}e_b^{1/4}(\sqrt{n}\E(Z))^{3/2},$$
$$(\sqrt{n}\E(Z))^{1/2}e_b^{3/4}\leq \frac{\eta}3e_b^{1/4}(\sqrt{n}\E(Z))^{3/2}+\frac{2e_b}{3\sqrt{\eta}}.$$
Thus
\begin{eqnarray*}
\E(Z^2{\bf 1}_{Z\geq\E(Z)})&\leq& \left(2v^2+\frac29e_b+v\frac{\sqrt{2\pi e_b}}2\right)\frac{1}{n}+\sqrt{\pi}\frac{\sqrt{\sqrt{n}\E(Z)}\left(e_b\right)^{3/4}}n\\
&&+\left(\frac{14}3\sqrt{e_b}+v\sqrt{2\pi}\right)\frac{\sqrt{n}\E(Z)}n+2\sqrt{\pi}\frac{(\sqrt{n}\E(Z))^{3/2}\left(e_b\right)^{1/4}}n\\
&\leq& \left(2+\eta\frac{\sqrt{2\pi}}4\right)\frac{v^2}n+\sqrt{\frac{2\pi}{n}}v\E(Z)+\left(\frac29+\frac{\sqrt{2\pi}}{4\eta}+\frac{2\sqrt{\pi}}{3\sqrt{\eta}}+\frac{14}{9\eta^2}\right)\frac{e_b}n\\
&&+\left(\eta\left(\frac{\sqrt{\pi}}3+\frac{28}9\right)+2\sqrt{\pi}\right)\frac{(\sqrt{n}\E(Z))^{3/2}\left(e_b\right)^{1/4}}n.
\end{eqnarray*}
Therefore, taking $\eta=0.088$, we obtain
\begin{equation*}
\E(Z^2{\bf 1}_{Z\geq\E(Z)})\leq2.1\frac{v^2}n+15^2\frac{e_b}n+\sqrt{2\pi}v\frac{\sqrt{n}\E(Z)}{n}+\sqrt{15}\frac{(\sqrt{n}\E(Z))^{3/2}\left(e_b\right)^{1/4}}n.
\end{equation*}
Finally, we use Cauchy-Schwarz inequality to obtain that $\sqrt{n}\E(Z)\leq(n\E(Z^2))^{1/2}=(D)^{1/2}$. Since $v^2\leq D$, we get (\ref{E1}).
\end{proof}
We deduce from this result the following concentration inequalities for $Z^2$
\begin{coro}\label{C3}
Let $e_b=b^2/n$. We have, for all $x>0$,
\begin{equation*}
\p\left(Z^2-\frac{D}n>\frac{D^{3/4}(e_b (19x)^2)^{1/4}+3\sqrt{Dv^2x})+3v^2x+e_b (19x)^2}n\right)\leq e^{-x}.
\end{equation*}
Moreover, for all $x>0$, with probability larger than $1-e^{-x}$,
\begin{equation}\label{ed}
\frac{D}n-Z^2\leq\frac{D^{3/4}e_b^{1/4}(\sqrt{15}+4.127\sqrt{x})+\sqrt{v^2D}(4.61+3\sqrt{x})+225e_b(6.2x^2+1)}n.
\end{equation}
\end{coro}
\begin{proof}:
From Bousquet's version of Talagrand's inequality and from $(\E(Z))^2\leq\E(Z^2)$, we obtain that, for all $x>0$, with probability larger than $1-e^{-x}$, $Z^2-D/n$ is not larger than
\begin{equation*}
\frac{4D^{3/4}(e_b x^2)^{1/4}+\sqrt{D}(14\sqrt{e_b x^2}/3+2\sqrt{2v^2x})+4D^{1/4}(e_b x^2)^{3/4}/3+3v^2x+e_b x^2/3}n.
\end{equation*}
We use repeatedly the inequality $a^{\alpha}b^{1-\alpha}\leq \alpha a+(1-\alpha)b$ to obtain that, with probability at least $1-e^{-x}$,  $Z^2-D/n$ is not larger than
\begin{equation*}
\frac{(4+32\eta/9)D^{3/4}(e_b x^2)^{1/4}+2\sqrt{2}\sqrt{Dv^2x}+3v^2x+(3+14/\eta^2+8/\sqrt{\eta})e_b x^2/9}n.
\end{equation*}
For $\eta=0.07$, this gives
\begin{equation*}
Z^2-\frac{D}n>\frac{D^{3/4}(e_b (19x)^2)^{1/4}+2\sqrt{2}\sqrt{Dv^2x}+3v^2x+e_b (19x)^2}n.
\end{equation*}
For the second one we use Klein's version of Talagrand's inequality to obtain, for all $x>0$ such that $r(x)=\sqrt{2(v^2+2b\E(Z))x/n}+8bx/3n<\E(Z)$,
\begin{equation*}
\p\left(Z^2<\left(\E(Z)-r(x)\right)^2\right)\leq e^{-x}.
\end{equation*}
We have $\left(\E(Z)-r(x)\right)^2=(\E(Z))^2-2\E(Z)r(x)+r(x)^2\geq (\E(Z))^2-2\E(Z)r(x)$, thus 
\begin{equation*}
\p\left(Z^2<(\E(Z))^2-2\E(Z)r(x)\right)\leq e^{-x}.
\end{equation*}
From the previous corollary, $(\E(Z))^2\geq\E(Z^2)-E_m$, thus 
\begin{equation*}
\p\left(Z^2<\E(Z^2)-E_m-2\E(Z)r(x)\right)\leq e^{-x}.
\end{equation*}
In order to conclude the proof of \ref{C3}, just remark that
\begin{eqnarray*}
2\E(Z)r(x)&\leq&\frac{4D^{3/4}(e_b x^2)^{1/4}+3\sqrt{Dv^2x}+16\sqrt{De_b x^2}/3}n\\
&\leq&\frac{(4+32\eta/9)D^{3/4}(e_b x^2)^{1/4}+3\sqrt{Dv^2x}+16/(9\eta^2)e_b x^2}n.
\end{eqnarray*}
For $\eta=0,0357$, we obtain (\ref{ed}).
\end{proof}
Finally, we have obtained the following result for the concentration of $Z^2$ around its mean
\begin{coro} \label{C4}
For all $x>0$,
\begin{equation*}
\p\left(Z^2-\frac{D}n>\frac{D^{3/4}(e_b (19x)^2)^{1/4}+3\sqrt{Dv^2x}+3v^2x+e_b (19x)^2}n\right)\leq e^{-x}.
\end{equation*}
\begin{equation*}
\p\left(Z^2-\frac{D}n<-\frac{8D^{3/4}(e_b x^2)^{1/4}+7.61\sqrt{v^2Dx}+e_b (40.25x)^2}n\right)\leq ee^{-x}.
\end{equation*}
\end{coro}
\begin{proof}:
In order to obtain the second inequality, we remark that the inequality is trivial when $x\leq 1$, thus we only have to use (\ref{ed}) for $x>1$ and then $\sqrt{x}>1$ and $x^2>1$.
\end{proof}
We will use this lemma to obtain a concentration inequality for totally degenerate $U$-statistics of order 2. The following result generalizes a previous inequality due to Houdr\'e $\&$ Reynaud-Bouret \cite{HRB03} to random variables taking values in a measurable space.

\begin{lemma}\label{TL1}
Let $X,X_1,...,X_n$ be i.i.d random variables taking value in a measurable space $(\x,\mathcal{X})$ with common law $P$. Let $\mu$ be a measure on $(\x,\mathcal{X})$ and let $(t_{\lambda})_{\lambda\in\Lambda}$ be a set of functions in $L^2(\mu)$. Let $$B=\{t=\sum_{\lambda\in\Lambda}a_{\lambda}t_{\lambda},\;\sum_{\lambda\in\Lambda}a_{\lambda}^2\leq 1\},\;D=\E\left(\sup_{t\in B}(t(X)-Pt)^2\right),$$ 
$$v^2=\sup_{t\in B}\Var(t(X)),\;b=\sup_{t\in B}\left\|t\right\|_{\infty}\;\textrm{and}\;e_b=\frac{b^2}n.$$ 
Let
\begin{equation*}
U=\frac1{n(n-1)}\sum_{i\neq j=1}^n\sum_{\lambda\in\Lambda}(t_{\lambda}(X_i)-Pt_{\lambda})(t_{\lambda}(X_j)-Pt_{\lambda}).
\end{equation*}
Then the following inequality holds
\begin{equation}\label{Uin1}
\forall x>0,\;\p\left( U>\frac{5.31D^{3/4}(e_b x^2)^{1/4}+3\sqrt{v^2Dx}+3v^2x+e_b (19.1x)^2}{n-1}\right)\leq2e^{-x}.
\end{equation}
\begin{equation}\label{Uin2}
\forall x>0,\;\p\left( U<-\frac{9D^{3/4}(e_b x^2)^{1/4}+7.61\sqrt{v^2Dx}+e_b (40.3x)^2}{n-1}\right)\leq3.8e^{-x}.
\end{equation}
\end{lemma}
\begin{proof}:
Remark that, from Cauchy-Schwarz inequality,
\begin{equation*}
\sup_{t\in B}(\nu_n(t))^2=\left(\sup_{\sum a_{\lambda}^2\leq 1}\sum_{\lambda\in\Lambda}a_{\lambda}\nu_n(t_{\lambda})\right)^2=\sum_{\lambda\in\Lambda}(\nu_n(t_{\lambda}))^2.
\end{equation*}
For all $x$ in $\x$, from Cauchy-Schwarz inequality, 
$$\sup_{t\in B}(t(x)-Pt)^2=\sum_{\lambda}(t_{\lambda}(x)-Pt_{\lambda})^2,$$
in particular, $D=\sum_{\lambda\in\Lambda}\Var(\psi_{\lambda}(X)).$
Moreover, easy algebra leads to 
\begin{eqnarray*}
\sum_{\lambda\in\Lambda}(\nu_n(t_{\lambda}))^2&=&\frac{1}{n^2}\sum_{i=1}^n\sum_{\lambda\in\Lambda}(t_{\lambda}(X_i)-Pt_{\lambda})^2\\
&&+\frac{1}{n^2}\sum_{i\neq j=1}^n\sum_{\lambda\in\Lambda}(t_{\lambda}(X_i)-Pt_{\lambda})(t_{\lambda}(X_j)-Pt_{\lambda})\\
&=&\frac 1nP_n\left(\sum_{\lambda\in\Lambda}(t_{\lambda}-Pt_{\lambda})^2\right)+\frac{n-1}nU.
\end{eqnarray*}
Let $Z^2=\sup_{t\in B}(\nu_n(t))^2$, $T_{\Lambda}=\sum_{\lambda\in\Lambda}(t_{\lambda}-Pt_{\lambda})^2$, 
$$\E(Z^2)=\E\left(\frac 1nP_nT_{\Lambda}\right)=\frac{D}n.$$
Hence
$$U=\frac{n}{n-1}\left(Z^2-\E(Z^2)-\frac1n\nu_n(T_{\Lambda})\right).$$
From Corollary \ref{C4}, for all $x>0$,
$$\p\left(Z^2-\frac{D}n>\frac{D^{3/4}(e_b (19x)^2)^{1/4}+3\sqrt{v^2Dx}+3v^2x+e_b (19x)^2}n\right)\leq e^{-x}.$$
$$\p\left(Z^2-\frac{D}n<-\frac{8D^{3/4}(e_b (x)^2)^{1/4}+7.61\sqrt{v^2Dx}+e_b (40.25x)^2}n\right)\leq 2.8e^{-x}.$$
Moreover, from Bernstein inequality, for all $x>0$, 
$$\p\left(-\nu_nT_{\Lambda}>\sqrt{2De_b x}+\frac{e_b x}{3}\right)\leq e^{-x}.$$
$$\p\left(\nu_nT_{\Lambda}>\sqrt{2De_b x}+\frac{e_b x}{3}\right)\leq e^{-x}.$$
We apply inequality (\ref{**}) with $a=D^{3/4}(e_b x^2)^{1/4}$, $b=e_b \sqrt{x}$, $\alpha=2/3$ and we obtain
$$\p\left(-\nu_nT_{\Lambda}>\frac{2\sqrt{2}}3D^{3/4}(e_b x^2)^{1/4}+e_b\left(\frac{x+\sqrt{2x}}{3}\right)\right)\leq e^{-x}.$$
$$\p\left(\nu_nT_{\Lambda}>\frac{2\sqrt{2}}3D^{3/4}(e_b x^2)^{1/4}+e_b\left(\frac{x+\sqrt{2x}}{3}\right)\right)\leq e^{-x}.$$
Therefore, for all $x>0$,
$$\p\left(U>\frac{5.31D^{3/4}(e_b x^2)^{1/4}+3\sqrt{v^2Dx}+3v^2x+e_b \left((19x)^2+(x+\sqrt{2x})/3\right)}{n-1}\right)\leq2e^{-x}.$$
$$\p\left(U<-\frac{9D^{3/4}(e_b x^2)^{1/4}+7.61\sqrt{v^2Dx}+e_b \left((40.25x)^2+(x+\sqrt{2x})/3\right)}{n-1}\right)\leq3.8e^{-x}.$$
These inequalities are trivial when $x<1$. We only use them when $x>1$ and we obtain (\ref{Uin1}) and (\ref{Uin2}) since $x<x^2$ and $\sqrt{x}<x^2$ when $x>1$.
\end{proof}
Let us now state the corollary of Bernstein's inequality that we used repeatedly in the article.
\begin{lemma}\label{TL2}
Let $X,X_1,...,X_n$ be i.i.d random variables taking value in a measurable space $(\x,\mathcal{X})$ with common law $P$. Let $\mu$ be a measure on $(\x,\mathcal{X})$ and let $(\psi_{\lambda})_{\lambda\in\Lambda}$ be an orthonormal system in $L^2(\mu)$. Let $L$ be a linear functional in $L^2(\mu)$ and let $B=\{t=\sum_{\lambda\in\Lambda}a_{\lambda}L(\psi_{\lambda}),\;\sum_{\lambda\in\Lambda}a_{\lambda}^2\leq 1\}$, $v^2=\sup_{t\in B}\Var(t(X))$, $b=\sup_{t\in B}\left\|t\right\|_{\infty}$ and $e_b=b^2/n$. Let $u$ be a function in $S$, the linear space spanned by the functions $(\psi_{\lambda})_{\lambda\in\Lambda}$ and let $\eta>0$.
Then the following inequality holds
\begin{equation}\label{Uin}
\forall x>0,\;\p\left( \nu_n(L(u))>\frac{\eta}2\|u\|^2+\frac{2v^2x+e_b x^2/9}{\eta n}\right)\leq e^{-x}.
\end{equation}
\end{lemma}
\begin{proof}:
From Bernstein's inequality, 
$$\forall x>0,\;\p\left(\nu_n(L(u))>\sqrt{\frac{2\Var(L(u)(X))x}{n}}+\frac{\left\|L(u)\right\|_{\infty}x}{3n}\right)\leq e^{-x}.$$
Since $t=L(u/\|u\|)$ belongs to $B$,
\begin{eqnarray*}
\sqrt{\frac{2\Var(L(u)(X))x}{n}}+\frac{\left\|L(u)\right\|_{\infty}x}{3n}&=&\|u\|\left(\sqrt{\frac{2\Var(t(X))x}{n}}+\frac{\left\|t\right\|_{\infty}x}{3n}\right)\\
&\leq&\frac{\eta}2\|u\|^2+\frac1{2\eta}\left(\sqrt{\frac{2v^2x}{n}}+\frac{bx}{3n}\right)^2.
\end{eqnarray*}
We conclude the proof using the inequality $(a+b)^2\leq 2a^2+2b^2.$
\end{proof}

\bibliographystyle{plain}
\bibliography{bibliolerasle}
\end{document}